\documentclass[10pt]{article}

\usepackage[tbtags]{amsmath}
\usepackage{cases}
\usepackage{mathrsfs}
\usepackage{amsfonts}
\usepackage{amsmath,amsthm,amssymb}

\usepackage{color}

\newtheorem{lemma}{Lemma}[section]
\newtheorem{theorem}{Theorem}[section]

\newtheorem{proposition}{Proposition}[section]

\numberwithin{equation}{section}

\newcommand{\R}{\mathbb{R}}

\newcommand{\pa}{\partial}

\theoremstyle{plain}

\newif \ifLastSection \LastSectionfalse

\numberwithin{equation}{section}

\begin{document}

\title{{\bf {Convergence rate of solutions toward stationary
solutions to the isentropic micropolar fluid model
in a half line }}}

\author{Haiyan Yin\thanks{Corresponding author. School of Mathematical Sciences, Huaqiao University, Quanzhou
362021, P. R. China. Email: yinhaiyan2000@aliyun.com} }

\date{}

\maketitle

\begin{abstract}
   In this paper, we study the asymptotic behavior of solutions to
the initial boundary value problem for the one-dimensional compressible isentropic
 micropolar fluid model in a half line $\mathbb{R}_{+}:=(0,\infty).$ We mainly investigates the unique existence, the asymptotic stability and convergence rates of stationary solutions to the outflow problem for this model.
We obtain the convergence rates of global
solutions towards corresponding stationary solutions if the initial
perturbation belongs to the weighted Sobolev space.
The proof is based on the weighted energy method by
taking into account the effect of the microrotational velocity on the viscous compressible fluid.

\end{abstract}

\medskip


\medskip


\section{Introduction}
The 1-D compressible viscous micropolar fluid model in the half line
$\R_{+}=:(0,+\infty)$ reads in Eulerian coordinates:
\begin{eqnarray}\label{NSP1*}
&&\left\{\begin{aligned}
& \rho_{t}+(\rho u)_{x}=0,\\
&(\rho u)_{t}+(\rho u^{2})_{x}+p(\rho)_x
=\lambda u_{xx},\\
&(\rho\omega)_{t}+(\rho u\omega)_{x}+\mu\omega
=\nu\omega_{xx}.
\end{aligned}\right.
\end{eqnarray}
Here the unknown functions $\rho$, $u$ and $\omega$ represent the density, the velocity and microrotational velocity, respectively. The pressure $p(\rho)=K \rho^{\gamma}$, with the adiabatic exponent $\gamma\geq 1$
and the gas constant $K>0$. The positive constants $\lambda$, $\mu$ and $\nu$ are the viscosities. The model of micropolar
fluid was first introduced by Eringen \cite{A. C. Eringen} in 1966. This model can be
used to describe the motions of a large variety of complex
fluids consisting of dipole elements such
as the suspensions, animal blood, liquid crystal, etc. For more physical background on this model,
we refer to \cite{A. C. Eringen2, G. Lukaszewicz}.

Much attention has been paid to the compressible micropolar
fluid model by many mathematicians in the last several decades.
For the isentropic case, Chen in \cite{Partial} investigated the global existence of strong solutions to the one-dimensional compressible micropolar fluid model with initial vacuum.
Later, Chen and his collaborators in \cite{Commun.} further studied
 the global weak solutions to the compressible micropolar fluid model with discontinuous initial data and vacuum. For three-dimensional compressible micropolar fluid model, the optimal decay rate in $L^{2}$ norm was studied by Liu and Zhang in \cite{ZhangLiu, Q.Q. Liu} with or without an external force. Later,
Wu and Wang \cite{wuzhigang} derived the pointwise estimates of the solution to the compressible micropolar fluid model and extended the optimal $L^{2}$ decay rate in
\cite{Q.Q. Liu} to the $L^{p}$ optimal decay rate with $p>1$. We also mention that there have
been many results on the incompressible micropolar
fluid system, see \cite{chen1, Nowakowsk, Dong} and the references
therein.

For the non-isentropic case, Mujakovi$\acute{c}$ first investigated the one-dimensional compressible micropolar
fluid model
 and obtained a series of results concerning the local-in-time existence, the
global existence and the regularity of solutions to an initial-boundary value problem with homogeneous
\cite{N. Mujakovi,global,regularity} and non-homogeneous \cite{local, Math.
Inequal, boundary,Nermina} boundary conditions.
The authors in \cite{L. Huang1,estimates05,Proceedings,Cauchy,stabilization} studied the large time
behavior of the solutions and the regularity of solutions to initial-boundary value problem and the Cauchy problem
of the one-dimensional compressible micropolar
fluid model. Duan in \cite{Duancal11, Duancal} investigated the global existence of strong solutions for the one-dimensional compressible micropolar fluids.
Chen and his collaborators in \cite{M.
Chen,J. Zhang} studied the blow up criterion of strong solutions to the
three-dimensional compressible micropolar fluid model.
 Dra$\check{z}$i$\acute{c}$ and Mujakovi$\acute{c}$ in \cite{spherical, spherical22} studied the regularity and large-time behavior of the spherical symmetry solutions for the three-dimensional compressible micropolar
fluid model.

Recently, there have been a series of mathematical results in the literature to study of the stability of wave patterns for the compressible nonisentropic micropolar fluid model: Liu and Yin \cite{Qcontac} for stability of contact discontinuity for the Cauchy problem; Jin and Duan \cite{JiJing} for the stability of rarefaction waves for the Cauchy problem; Yin \cite{stabstationary}
for stability of stationary solutions for the inflow problem; Cui and Yin \cite{Cui1} for stability of composite waves for the inflow problem;  Cui and Yin \cite{Cui2} for the stability and convergence rate of stationary solutions for the outflow problem. However, the articles mentioned above about the stability of waves are all based on the assumption of microrotation velocity $\omega=0$ for the
large time behavior. In this paper, we expect to study the asymptotic stability of
stationary solutions to the one-dimensional compressible isentropic micropolar fluid model
for the outflow problem without the assumption of $\omega=0$ for the
large time behavior.

Initial data for system \eqref{NSP1*} is given by
\begin{equation}
\label{dd-id}
(\rho,u,\omega)(x,0)=(\rho_{0},u_{0},\omega_{0})(x),\quad
\inf_{x\in \mathbb{R}_{+}}\rho_{0}(x)>0.
\end{equation}
We assume that the initial data at the far field $x=+\infty$ is
constant, namely
\begin{equation}
\label{dd-ff}
\lim\limits_{x\rightarrow+\infty}(\rho_{0},u_{0},\omega_{0})(x)=(\rho_{+},u_{+},\omega_{+}),
\end{equation}
In particular, to construct a classical solution of the micropolar fluid model
 $\eqref{NSP1*}_{3}$,  it is necessary to require that
\begin{eqnarray}\label{1.4(b)}
&\omega_{+}=0.
\end{eqnarray}

The boundary data for $u$ and $\omega$ at $x=0$ is
given by
\begin{equation}
\label{ddw1-bd} (u,\omega)(0,t)=(u_{b},\omega_{b}), \quad
\forall\, t\geq 0,
\end{equation}
where $u_{b}<0$, $\omega_{b}\neq 0$ are constants and the following
compatibility conditions hold
\begin{equation}
\label{compa} u_{0}(0)=u_{b},\quad \omega_{0}(0)=\omega_{b}.
\end{equation}

The assumption $u_{b}<0$ means that fluid blows out from the
boundary $x=0$ with the velocity $u_{b}$. Thus this
problem is called an outflow problem (see \cite{Matsumura1}). The
outflow boundary condition implies that the characteristic of the
hyperbolic equation $\eqref{NSP1*}_{1}$ for the density $\rho$ is
negative around the boundary so that boundary
conditions on $u$ and $\omega$ to parabolic
equations $\eqref{NSP1*}_{2}$ and
$\eqref{NSP1*}_{3}$ are necessary and sufficient for the wellposedness of this problem.

We guess that the large time behavior of solutions to the
initial boundary value problem \eqref{NSP1*}, \eqref{dd-id},
\eqref{dd-ff}, \eqref{ddw1-bd}, \eqref{compa} are the stationary
solutions to $\eqref{NSP1*}$
 independent of a time variable t
\begin{eqnarray}\label{1.6}
&&\left\{\begin{aligned}
& (\tilde{\rho} \tilde{u})_{x}=0,\\
& (\tilde{\rho}\tilde{u}^{2})_{x}+p(\tilde{\rho})_{x}=\lambda\tilde{u}_{xx},\\
& (\tilde{\rho} \tilde{u}\tilde{\omega})_{x}+\mu\tilde{\omega}
=\nu\tilde{\omega}_{xx},
\end{aligned}\right.
\end{eqnarray}
with the boundary data
\begin{eqnarray}\label{1.7}
&&\begin{aligned} &\inf_{x\in \mathbb{R}_{+}}\tilde\rho(x)>0,\ \
\lim_{x\rightarrow\infty}(\tilde\rho,\tilde u, \tilde
\omega)(x)=(\rho_{+},u_{+},0),\ \ \tilde u(0)=u_{b}<0, \ \ \ \tilde\omega(0)=\omega_{b}\neq 0.
\end{aligned}
\end{eqnarray}
Integrating  $\eqref{1.6}_{1}$  over $(x,\infty)$  for $x>0$ yields
\begin{eqnarray}\label{1.7(a)}
\tilde \rho(x)\tilde u(x)=\rho_{+}u_{+},
\end{eqnarray}
which implies by letting $x\rightarrow 0^{+}$,
\begin{eqnarray}\label{1.7(b)}
u_{+}=\frac{\tilde\rho(0)\tilde u(0)}{\rho_{+}}=\frac{\tilde\rho(0)u_{b}}{\rho_{+}}<0.
\end{eqnarray}

Now define \begin{eqnarray}\label{inequalityds21g6}
\tilde{\delta}=\max\left\{|\omega_{b}|,\ |u_{b}-u_{+}|\right\},\ \ \ \sigma=\min\left\{|r_{1}|,\ \xi_{0}\right\},
\end{eqnarray}
where constants $r_{1}$ and $\xi_{0}$ are defined in Section 2.

 Main results of the present paper are stated in the following theorems.
\begin{theorem}\label{stationary}
The boundary value problem \eqref{1.6}-\eqref{1.7} has a unique
smooth solution $(\tilde{\rho}, \tilde{u}, \tilde{\omega} )$  if and only if $M_{+}\geq1$
and $\chi_{c}u_{+}>u_{b}$.

(i) If $M_{+}>1$, then the solution
$(\tilde{\rho}, \tilde{u}, \tilde{\omega} )$ satisfies the estimates
\begin{eqnarray}\label{inequality21}
|\partial_{x}^{k}(\tilde{\rho}-\rho_{+},\tilde{u}-u_{+}, \tilde{\omega}-\omega_{+} )|\leq C\tilde{\delta} e^{-\sigma x},\ \  for\   k=0,1,2,\cdots,
\end{eqnarray}
where $C$ and $\sigma$ are positive constants.

(ii) If $M_{+}=1$, then the solution
$(\tilde{\rho}, \tilde{u}, \tilde{\omega} )$ satisfies the estimates
\begin{eqnarray}\label{inequality21g6}
|\partial_{x}^{k}(\tilde{\rho}-\rho_{+},\tilde{u}-u_{+} )|\leq C\frac{\tilde{\delta}^{k+1}}{\left(1+\tilde{\delta} x\right)^{k+1}},\ \ \
|\partial_{x}^{k}(\tilde{\omega}-\omega_{+} )|\leq C\tilde{\delta}e^{-\sigma x}
\ \  for\   k=0,1,2,\cdots,
\end{eqnarray}
where $C$ and $\sigma$ are positive constants.
\end{theorem}

\begin{theorem}\label{1.2theorem}
Suppose that stationary solution
$(\tilde{\rho},\tilde{u},\tilde{\omega})$ exists.

(i)Assume that $M_{+}>1$ holds. For an arbitrary  positive constant
 $\vartheta$, there exist positive constants $\beta$ and
$\varepsilon_{1}$ such that if $(1+\beta
x)^{\frac{\vartheta}{2}}(\rho_{0}-\tilde{\rho}), (1+\beta
x)^{\frac{\vartheta}{2}}(u_{0}-\tilde{u}), (1+\beta
x)^{\frac{\vartheta}{2}}(\omega_{0}-\tilde{\omega})$ respectively
belongs to the Lebesgue space  $L^2(\mathbb{R}_{+})$ and
$\|[\rho_{0}-\tilde{\rho},u_{0}-\tilde{u},\omega_{0}-\tilde{\omega}]\|_{1}+\beta+\tilde{\delta}\leq
\varepsilon_{1},$
 then the initial boundary value problem
\eqref{NSP1*}, \eqref{dd-id}, \eqref{dd-ff}, \eqref{ddw1-bd},
\eqref{compa} has a unique solution $[\rho,u,\omega]$
verifying the decay estimate
\begin{equation}\label{1.thre1b}
\begin{aligned}
\|[\rho-\tilde{\rho},u-\tilde{u},\omega-\tilde{\omega}](t)\|_{\infty}
 \leq C(1+t)^{-\frac{\vartheta}{2}}.
\end{aligned}
\end{equation}

(ii)Assume that $M_{+}=1$ holds. There exists a positive constant
$\varepsilon_{2}$ such that if the initial data
satisfies $\|[(1+Bx)^{\frac{\vartheta}{2}}(\rho_{0}-\tilde{\rho}), (1+Bx)^{\frac{\vartheta}{2}}(u_{0}-\tilde{u}), (1+Bx)^{\frac{\vartheta}{2}}(\omega_{0}-\tilde{\omega})]\|_{1}
\leq \varepsilon_{2}$ for positive constants $B$ and $\vartheta$
satisfying $\vartheta\in[2,\vartheta^{*})$, where $B$ and
$\vartheta^{*}$ is respectively defined by  \eqref{stationary1} and
\begin{equation}\label{1.thre1b}
\begin{aligned}
\vartheta^{*}(\vartheta^{*}-2)=\frac{4}{\gamma+1},\ \ and \ \ \vartheta^{*}>0,
\end{aligned}
\end{equation}
then the initial boundary value problem
\eqref{NSP1*}, \eqref{dd-id}, \eqref{dd-ff}, \eqref{ddw1-bd},
\eqref{compa} has a unique solution $(\rho,u,\omega)$
verifying the decay estimate
\begin{equation}\label{1.thre1}
\begin{aligned}
\|[\rho-\tilde{\rho},u-\tilde{u},\omega-\tilde{\omega}](t)\|_{\infty}
 \leq C(1+t)^{-\frac{\vartheta}{4}}.
\end{aligned}
\end{equation}
\end{theorem}

 If the microstructure of the fluid is not taken into account, that
is to say the effect of the microrotational velocity is omitted, i.e., $\omega=0$, then equations
\eqref{NSP1*} reduce to the classical Navier-Stokes equations.
 So far, there have been a great number of mathematical
studies about the outflow problem,  impermeable wall problem and
inflow problem for Navier-Stokes equations, please referring to
\cite{HQad,KNZhd,HfklgQadcd,HfgQad,NNY,M. Nishikawnj, Math.} and the
references therein. Under the assumption of microrotation velocity $\omega=0$ for the
large time behavior, there also have been a series of mathematical results about the above mentioned problems for the compressible nonisentropic micropolar fluid model, see \cite{stabstationary, Cui1, Cui2} and the references therein. Theorem \ref{stationary} shows
the unique existence of stationary solutions without the assumption of $\omega=0$ for the
large time behavior. We point out that the stationary equation \eqref{1.6} can be divided into two independent stationary equations. One is the equation \eqref{1.6cdss} which is also the stationary equation of the isentropic Navier-Stokes equation.
The other is the equation \eqref{1.6cv} about the microrotation velocity $\omega$ which is a second order homogeneous linear ordinary differential equation with constant coefficients.
The detailed process about stationary solutions can be seen in Section 2.
Theorem \ref{1.2theorem} shows that the time asymptotic
stability and convergence rates of stationary solutions for the compressible isentropic micropolar fluid model
under smallness assumptions on the boundary data and the initial
perturbation in the Sobolev space by employing the weighted energy
method, and particularly, the  microrotational velocity $\omega(x,t)$ has the nontrivial large-time behavior. Compared to the classical Navier-Stokes system without any force, the main difficulty in the proof for the micropolar fluid model is to treat the estimates on those terms related to the microrotational velocity $\omega(x,t)$ as mentioned above.
To the best of our knowledge, this is the first
work on the stability of stationary solutions for the compressible micropolar
fluid model without the assumption of $\omega=0$ for the
large time behavior.

\medskip

The rest of the paper is arranged as follows.  In
Section 2, we prove the existence of the stationary solution.
In the main part Section 3, we give the {\it a priori} estimates on the solutions of
the perturbative equations for the supersonic case $M_{+}>1$ and transonic case $M_{+}=1$, respectively.
The proof of Theorem \ref{1.2theorem} is concluded in Section 4.

\medskip

\textbf{Notation:} Throughout the paper, we denote positive
constants (generally large) and (generally small) independent of $t$
by $C$ and $c$, respectively. And the character ``$C$" and ``$c$"
may take different values in different places.
 $L^p
= L^p(\mathbb{R}_{+}) \ (1 \leq p \leq\infty)$ denotes
 the usual Lebesgue space on $[0,\infty)$ with its norm $ \|\cdot\|_ {L^p}$, and when $p=2$, we
write $ \| \cdot \| _{ L^2(\mathbb{R}_{+}) } = \| \cdot \|$.  $
H^s(\mathbb{R}_{+})$ denotes the usual $s$-th order Sobolev space
with its norm $ \| f \|_{ H^s(\mathbb{R}_{+})}=\| f \|_s =( \sum
\limits _{i=0}^s\|\partial^i f \| ^2)^{\frac{1}{2}}$.
 A norm with algebraic
weight is defined as follows:
$$\|f\|_{\alpha,\beta,i}:=\left(\int W_{\alpha,\beta}\sum_{j\leq
i}(\partial^{j}f)^{2}dx\right)^{\frac{1}{2}}, \ \ \ i,j\in
\mathbb{Z}, \ \ i,j\geq 0,$$
$$W_{\alpha,\beta}:=(1+\beta x)^{\alpha},  \ \ \  \alpha>0.$$
Note that this norm is equivalent to the norm defined by $\|(1+\beta
x)^{\frac{\alpha}{2}}f\|_{i}.$ The
 last subscript $i$ is often dropped for the case of $i=0$, i.e.
 $\|f\|_{\alpha,\beta}:=\|f\|_{\alpha,\beta,0}$.

\section{Existence of the stationary solution }
In this section, we will prove the existence of the stationary
solution to the boundary value problem \eqref{1.6}-\eqref{1.7},
which is stated in Theorem \ref{stationary}.
Notice that the stationary problem \eqref{1.6}-\eqref{1.7} can be divided into the following two independent stationary equations
\begin{eqnarray}\label{1.6cdss}
&&\left\{\begin{aligned}
& (\tilde{\rho} \tilde{u})_{x}=0,\\
& (\tilde{\rho}\tilde{u}^{2})_{x}+p(\tilde{\rho})_{x}=\lambda\tilde{u}_{xx},
\end{aligned}\right.
\end{eqnarray}
with the boundary data
\begin{eqnarray}\label{1.7a}
&&\begin{aligned} &\inf_{x\in \mathbb{R}_{+}}\tilde\rho(x)>0,\ \
\lim_{x\rightarrow\infty}(\tilde\rho,\tilde u)(x)=(\rho_{+},u_{+}),\ \ \tilde u(0)=u_{b}<0,
\end{aligned}
\end{eqnarray}
and
\begin{eqnarray}\label{1.6cv}
(\tilde{\rho} \tilde{u}\tilde{\omega})_{x}+\mu\tilde{\omega}
=\nu\tilde{\omega}_{xx}.
\end{eqnarray}
with the boundary data
\begin{eqnarray}\label{1.7b}
\lim_{x\rightarrow\infty}\tilde
\omega(x)=0,\ \ \ \  \tilde\omega(0)=\omega_{b}\neq 0.
\end{eqnarray}

Now we firstly prove the existence of the stationary
solution to the stationary problem \eqref{1.6cv}-\eqref{1.7b}.
From \eqref{1.6cv} and \eqref{1.7(a)}, we have
\begin{eqnarray}\label{1.1nm8}
\rho_{+}u_{+}\tilde{\omega}_{x}+\mu\tilde{\omega}
=\nu\tilde{\omega}_{xx},
\end{eqnarray}
which is  a second order homogeneous linear ordinary differential equation with constant coefficients. The characteristic equation of \eqref{1.1nm8} is given by
\begin{eqnarray}\label{1.1nm8dc}
\nu r^{2}-\rho_{+}u_{+}r-\mu
=0.
\end{eqnarray}
Then the roots of characteristic equation are
\begin{eqnarray}\label{1.1nm8dc}
r_{1}=\frac{\rho_{+}u_{+}-\sqrt{\rho_{+}^{2}u_{+}^{2}+4\nu\mu}}{2\nu}, \ \ \  \
r_{2}=\frac{\rho_{+}u_{+}+\sqrt{\rho_{+}^{2}u_{+}^{2}+4\nu\mu}}{2\nu}.
\end{eqnarray}
From \eqref{1.7(b)}, we deduce that $r_{1}<0$ and $r_{2}>0$.
Then from the theory of the homogeneous linear ordinary differential equation
([1]), we get the general solution of \eqref{1.1nm8} as follows
\begin{eqnarray}\label{1.1nm8bndc}
\tilde{\omega}(x)=C_{1}e^{r_{1}x}+C_{2}e^{r_{2}x},
\end{eqnarray}
where $C_{1}$ and $C_{2}$ are arbitrary constants.
Notice that $r_{1}<0$, $r_{2}>0$, and the boundary data
\eqref{1.7b}, arbitrary constants $C_{1}$ and $C_{2}$  can be
uniquely determined as follows
$C_{1}=\omega_{b}$ and $C_{2}=0$. Hence the particular solution of
\eqref{1.1nm8} with boundary condition \eqref{1.7b} is given by
\begin{eqnarray}\label{1.1vfnm8bndc}
\tilde{\omega}(x)=\omega_{b}e^{r_{1}x}.
\end{eqnarray}

 Here we should note
that the existence of the stationary problem
\eqref{1.6cdss}-\eqref{1.7a} has been proved by Kawashima, Nishibata and Zhu in \cite{KNZhd}.
 Here we summarize the results in \cite{KNZhd}.
We firstly define
\begin{eqnarray}\label{1.18}
\tilde{\chi}(x):=\frac{ \tilde u(x)}{u_{+}}=\frac{ \rho_{+}}{\tilde \rho(x)}>0.
\end{eqnarray}
Then integrating $\eqref{1.6cdss}_{2}$ over $(x,+\infty)$ and
substituting $\eqref{1.18}$ in the resulting equality, we get
\begin{eqnarray}\label{1.19}
\lambda u_{+}\tilde{\chi}_{x}=F(\tilde{\chi}),
\end{eqnarray}
where
\begin{eqnarray}\label{1.20}
F(\tilde{\chi}):=K\rho_{+}^{\gamma}(\tilde{\chi}^{-\gamma}-1)
+\rho_{+}u_{+}^{2}(\tilde{\chi}-1).
\end{eqnarray}
$\tilde{\chi}$ satisfies boundary conditions
\begin{eqnarray}\label{1.20po}
\tilde{\chi}(0)=\frac{ u_{b}}{u_{+}}>0,\ \ \ \lim_{x\rightarrow
+\infty}\tilde{\chi}(x)=1
\end{eqnarray}
which can be derived by  \eqref{1.7a} and \eqref{1.7(b)}.
We introduce Mach number $M_{+}$ at the far field
${x=+\infty}$: $M_{+}=\frac{|u_{+}|}{c_{+}},$ where
$c_{+}=\sqrt{p'(\rho_{+})}=\sqrt{K\gamma\rho_{+}^{\gamma-1}}$ is sound speed.

If $M_{+}>1$, the equation $F(\tilde{\chi})=0$ has the distinct two roots $\tilde{\chi}=1$ and $\tilde{\chi}=\chi_{c}$ satisfying $\chi_{c}<\tilde{\chi}(0)$.  If $M_{+}=1$, the equation $F(\tilde{\chi})=0$ admits only one root $\tilde{\chi}=\chi_{c}=1$.

\begin{lemma}\label{lem.Va}(See \cite{KNZhd})
The boundary value problem \eqref{1.19}-\eqref{1.20po} has a unique
smooth solution $\tilde{\chi}$ if and only if $M_{+}\geq1$
and $\chi_{c}<\tilde{\chi}(0)$. Moreover, if $\tilde{\chi}(0)\lessgtr 1$,
then $\tilde{\chi}_{x}\gtrless 0$.

(i) If $M_{+}>1$, then the solution
$\tilde{\chi}(x)$ satisfies the estimates
\begin{eqnarray}\label{inequality21}
|\partial_{x}^{k}(\tilde{\chi}(x)-1)|\leq C|\tilde{\chi}(0)-1|e^{-\xi_{0} x},\ \  for\   k=0,1,2,\cdots.
\end{eqnarray}
where $C$ and $\xi_{0}$ are positive constants.

(ii) If $M_{+}=1$, then the solution
$\tilde{\chi}$ is monotonically decreasing and satisfies the estimates
\begin{eqnarray}\label{inequality21g6nj}
|\partial_{x}^{k}(\tilde{\chi}(x)-1)|\leq C\frac{|\tilde{\chi}(0)-1|^{k+1}}{\left(1+|\tilde{\chi}(0)-1|x\right)^{k+1}},\ \  for\   k=0,1,2,\cdots.
\end{eqnarray}
\end{lemma}
Then we complete the proof of Theorem \ref{stationary} from  \eqref{1.1vfnm8bndc}, \eqref{1.18} and Lemma \ref{lem.Va}.

\section{Energy estimates}

To prove Theorem \ref{1.2theorem}, we use the energy method. Define
the perturbation as
$$
[\varphi, \psi, \zeta ](x,t)=[\rho-\tilde{\rho},u-\tilde{u},
\omega-\tilde{\omega}],
$$
then $[\varphi,\psi,\zeta](x,t)$ satisfies
\begin{eqnarray}\label{NSP.pt}
&&\left\{\begin{aligned}
&\pa_t\varphi+u\pa_x\varphi+\rho\pa_x\psi=-\pa_x\tilde{u}\varphi-\pa_x\tilde{\rho}\psi,\\
& \rho\left(\pa_t \psi+ u\pa_x\psi\right)+\pa_x
[p(\rho)-p(\tilde{\rho})]=\lambda\pa^{2}_x\psi-\pa_x\tilde{u}(\tilde{u}\varphi+\rho\psi),\\
&\rho(\pa_{t}\zeta+u\pa_{x}\zeta)+\mu\zeta
=\nu\pa_{x}^{2}\zeta-\pa_x\tilde{\omega}(\tilde{u}\varphi+\rho\psi),
\end{aligned}\right.
\end{eqnarray}
with initial data
\begin{eqnarray}\label{NSP.ptid}
\begin{array}{rl}
\left[\varphi,\psi,\zeta](x,0)\right. = &
\left.[\varphi_{0},\psi_{0},\zeta_{0}\right](x)=
[\rho_{0}(x)-\tilde{\rho}(x),u_{0}(x)-\tilde{u}(x),
\omega_{0}(x)-\tilde{\omega}(x)]
\end{array}
\end{eqnarray}
and boundary condition
\begin{equation}\label{NSfvP.ptid}
\psi(0,t)=\zeta(0,t)=0.
\end{equation}

 In the paper, to prove Theorem \ref{1.2theorem}, for
brevity we only devote ourselves to obtaining the global-in-time
{\it a priori} estimates in the following. We look for the solution
$[\varphi, \psi, \zeta](x,t)$ in the solution space $X([0,+\infty))$ which is defined as follows:
\begin{equation}\notag
\begin{split}
&X([0,T))=\bigg\{[\varphi, \psi,  \zeta]\big|[\varphi, \psi,  \zeta]\in C([0,T];H^1(\mathbb{R}_{+})),[\pa_{x}\varphi,\zeta]\in L^{2}([0,T];L^2(\mathbb{R}_{+})),\\
 &
\qquad\qquad\qquad\qquad\qquad\qquad\qquad\qquad\qquad\ \ \ \pa_{x}[\psi,\zeta]\in L^{2}([0,T];H^1(\mathbb{R}_{+}))\bigg\}
\end{split}
\end{equation}
for some $0<T\leq +\infty$. Lemma \ref{lem.V} plays an important role in the proof of the {\it a priori} estimates for supersonic case $M_{+}>1$ and transonic case $M_{+}=1$, respectively.

\begin{lemma}\label{lem.V}
(i) For any function  $h(\cdot, t)\in H^{1}(\R_{+}),$ there is a positive
constant C such that
\begin{eqnarray}\label{4.2}
\begin{aligned}
&\int_{\mathbb{R}_{+}}e^{-\sigma x}|h|^{2}dx \leq
C\left(h^{2}(0,t)+\|\partial_{x}h(t)\|^{2}\right).
\end{aligned}
\end{eqnarray}
\end{lemma}
(ii)\  Let $k>1.$ For any function  $h(\cdot, t)\in H^{1}(\R_{+}),$ there is a positive
constant C such that
\begin{eqnarray}\label{4.2cd}
\begin{aligned}
\int_{\mathbb{R}_{+}}\frac{\tilde{\delta}^{k+1}}{\left(1+\tilde{\delta} x\right)^{k+1}}h^{2}dx
\leq C\tilde{\delta}^{k}h^{2}(0,t)+C\tilde{\delta}^{k-1}\|\pa_xh\|^{2}.
\end{aligned}
\end{eqnarray}
\begin{proof}
(i) \ \eqref{4.2} can be derived from the following Poincar\'{e}
type inequality:
\begin{equation}\label{p.ine}
\begin{split}
|h(x,t)|\leq |h(0,t)|+x^{\frac{1}{2}}\|\pa_xh\|.
\end{split}
\end{equation}

(ii)\ Letting $k>1$ and using Poincar\'{e}
type inequality \eqref{p.ine}, we
compute as
\begin{eqnarray*}\label{4.nbjn2}
\begin{aligned}[b]
\int_{\mathbb{R}_{+}}\frac{\tilde{\delta}^{k+1}}{\left(1+\tilde{\delta} x\right)^{k+1}}h^{2}dx
\leq\int_{\mathbb{R}_{+}}\frac{\tilde{\delta}^{k+1}}{\left(1+\tilde{\delta} x\right)^{k+1}}(h^{2}(0,t)+x\|\pa_xh\|^{2})dx
\leq C\tilde{\delta}^{k}h^{2}(0,t)+C\tilde{\delta}^{k-1}\|\pa_xh\|^{2}.
\end{aligned}
\end{eqnarray*}
\end{proof}

\subsection{The a priori
estimates for $M_{+}>1$}
The key to the proof of our main Theorem
$\ref{1.2theorem}$ (i) is to derive the uniform  {\it a priori}
estimates of solutions to the initial boundary value problem
\eqref{NSP.pt}, \eqref{NSP.ptid} and \eqref{NSfvP.ptid}. For the convenience of stating the a priori assumption, we use the notations
\begin{eqnarray}\label{3.2}
 \begin{aligned}[b]
N_{1}(T):=\sup\limits_{0\leq t\leq
T}\|(\varphi,\psi,\zeta)(t)\|_{1}
\end{aligned}
\end{eqnarray}

\begin{proposition}\label{priori.est}
 Assume  the same conditions as in Theorem $\mathrm{\ref{1.2theorem}}$(i) hold.
 Let $\vartheta$, $\beta$ and $\kappa$ be  positive constants.
Suppose $[\varphi,\psi,\zeta]\in X([0,T))$ is a solution to
\eqref{NSP.pt}, \eqref{NSP.ptid} and \eqref{NSfvP.ptid} which
satisfies $(1+ \beta x)^{\frac{\vartheta}{2}}\varphi,(1+\beta
x)^{\frac{\vartheta}{2}}\psi, (1+ \beta
x)^{\frac{\vartheta}{2}}\zeta\in C([0,T];L^2(\mathbb{R}_{+}))$  for a certain
positive constant $T$. For arbitrary $\alpha\in[0,\vartheta]$, there
exist positive constants $C$ and $\varepsilon_{1}$ independent of
$T$ such that if
$N_{1}(T)+\tilde{\delta}+\beta\leq
\varepsilon_{1}$
is satisfied, it holds for an arbitrary $t\in[0,T]$ that
\begin{eqnarray}\label{3.3x}
 \begin{aligned}[b]
&(1+t)^{\alpha+\kappa}\|[\varphi,\psi,\zeta](t)\|_{1}^{2}
+\int_{0}^{t}(1+
\tau)^{\alpha+\kappa}\|[\zeta,\partial_{x}[\psi,\zeta],\partial_{x}^{2}[\psi,\zeta]](\tau)\|^{2}d\tau
\\
\leq & C (1+
t)^{\kappa}\left(\|[\varphi_{0},\psi_{0},\zeta_{0}]\|_{1}^{2}+\|[\varphi_{0},
\psi_{0},\zeta_{0}]\|_{\vartheta,\beta}^{2}\right).
\end{aligned}
\end{eqnarray}
\end{proposition}

\begin{lemma}\label{lem.V1bnm}
There exists a positive constant $\varepsilon_{1}$ such that if
$N_{1}(T)+\tilde{\delta}+\beta\leq\varepsilon_{1}$, then
\begin{eqnarray}\label{1m}
\begin{aligned}[b]
  &(1+
t)^{\xi}\|[\varphi,\psi,\zeta](t)\|_{\alpha,\beta}^{2}+\beta\int_{0}^{t}(1+
\tau)^{\xi}
\|[\varphi,\psi,\zeta](\tau)\|_{\alpha-1,\beta}^{2}d\tau\\&+\int_{0}^{t}(1+
\tau)^{\xi}\|[\zeta,\partial_{x}[\psi,\zeta]](\tau)\|_{\alpha,\beta}^{2}d\tau+\int_{0}^{t}(1+
\tau)^{\xi}\varphi^{2}(0,\tau)d\tau\\
 \leq &C\|[\varphi_{0},\psi_{0},\zeta_{0}]\|_{\vartheta,\beta}^{2} +\xi\int_{0}^{t}(1+
\tau)^{\xi-1}\|[\varphi,\psi,\zeta](\tau)\|_{\alpha,\beta}^{2}d\tau
+C\tilde{\delta}\int_{0}^{t}(1+
\tau)^{\xi}\|\partial_{x}\varphi\|^{2}(\tau)d\tau
\end{aligned}
\end{eqnarray}
holds for $\alpha \in [0,\vartheta]$ and $\xi\geq0$.
\end{lemma}

\begin{proof}
 A direct computation by using $\eqref{NSP1*}$ and
$\eqref{NSP.pt}_{1}$, we have the following identity
\begin{eqnarray}\label{3.7}
 \begin{aligned}[m]
[\rho\Phi(\rho,\tilde{\rho})]_{t}+[\rho
u\Phi(\rho,\tilde{\rho})]_{x}+(p(\rho)-p(\tilde{\rho}))\psi_{x}
+\tilde{u}_{x}[p(\rho)-p(\tilde{\rho})-p'(\tilde{\rho})\varphi]+\frac{p(\tilde{\rho})_{x}}{\tilde{\rho}}\varphi\psi
=0,
\end{aligned}
\end{eqnarray}
where
$\Phi(\rho,\tilde{\rho})=\displaystyle\int_{\tilde{\rho}}^{\rho}\frac{p(s)-p(\tilde{\rho})}{s^{2}}ds.$

It is easy to see that $\Phi(\rho,\tilde{\rho})$ is equivalent to
$|\varphi|^{2}$, i.e.,
 \begin{eqnarray}\label{3.7as}
c|\varphi|^{2}\leq\Phi(\rho,\tilde{\rho})\leq C|\varphi|^{2},
\end{eqnarray}
since there exist positive constants $c$ and $C$ such
that $\rho$ and $\tilde{\rho}$ satisfying
\begin{eqnarray}\label{3.7ahus}
0<c\leq\rho, \ \tilde{\rho} \leq C.
\end{eqnarray}
Multiply $\eqref{NSP.pt}_{2}$ by
$\psi$ and $\eqref{NSP.pt}_{3}$ by
$\zeta$, we have the following identity
\begin{eqnarray}\label{3.5}
 \begin{aligned}[b]
&(\frac{\rho}{2}\psi^{2}+\frac{\rho}{2}\zeta^{2})_{t}+(\frac{\rho u}{2}\psi^{2}+\frac{\rho u}{2}\zeta^{2})_{x}
+(p(\rho)-p(\tilde{\rho}))_{x}\psi-(\lambda\psi\psi_{x}+\nu\zeta\zeta_{x})_{x}+\mu\zeta^{2}+\lambda\psi_{x}^{2}+\nu\zeta_{x}^{2}\\
=&-\pa_x\tilde{u}[\tilde{u}\varphi\psi+\rho\psi^{2}]-\pa_x\tilde{\omega}(\tilde{u}\varphi\zeta+\rho\psi\zeta).
\end{aligned}
\end{eqnarray}
Combining  $\eqref{3.5}$ and  $\eqref{3.7}$, we have
\begin{eqnarray}\label{basic*}
\begin{aligned}[b]
&\left[\rho\Phi(\rho,\tilde{\rho})+\frac{\rho}{2}\psi^{2}+\frac{\rho}{2}\zeta^{2}\right]_{t}
+\left[\rho u\Phi(\rho,\tilde{\rho})+\frac{\rho
u}{2}\psi^{2}+\frac{\rho u}{2}\zeta^{2}+(p(\rho)-p(\tilde{\rho}))\psi-\lambda\psi\psi_{x}-\nu\zeta\zeta_{x}\right]_{x}\\&+\mu\zeta^{2}+\lambda\psi_{x}^{2}
+\nu\zeta_{x}^{2}
\\=&-\pa_x\tilde{u}[\tilde{u}\varphi\psi+\rho\psi^{2}+p(\rho)-p(\tilde{\rho})-p'(\tilde{\rho})\varphi]-\pa_x\tilde{\omega}(\tilde{u}\varphi\zeta+\rho\psi\zeta)-\frac{p(\tilde{\rho})_{x}}{\tilde{\rho}}\varphi\psi.
\end{aligned}
\end{eqnarray}

Multiplying \eqref{basic*} by $W_{\alpha,\beta}=(1+\beta x)^{\alpha}$ defined in Notation,
then we integrate the resulting equality over $\mathbb{R}_{+}$ to
get
\begin{equation}\label{z.eng1*}
\begin{aligned}[b]
\frac{d}{dt}&\int_{\R_{+}}W_{\alpha,\beta}\left[\rho\Phi(\rho,\tilde{\rho})+\frac{\rho}{2}\psi^{2}+\frac{\rho}{2}\zeta^{2}\right]dx
+\int_{\R_{+}}W_{\alpha,\beta}\left[\mu\zeta^{2}+\lambda\psi_{x}^{2}
+\nu\zeta_{x}^{2}\right]dx\\&
-\alpha\beta\int_{\R_{+}}
W_{\alpha-1,\beta}\left[\rho u\Phi(\rho,\tilde{\rho})+\frac{\rho
u}{2}\psi^{2}+\frac{\rho u}{2}\zeta^{2}+(p(\rho)-p(\tilde{\rho}))\psi\right]dx-\rho(0,t)u_{b}\Phi(\rho(0,t),\tilde{\rho}(0))\\
=&\underbrace{-\int_{\R_{+}}W_{\alpha,\beta}\left[\pa_x\tilde{u}(\tilde{u}\varphi\psi+\rho\psi^{2}+p(\rho)-p(\tilde{\rho})-p'(\tilde{\rho})\varphi)
+\pa_x\tilde{\omega}(\tilde{u}\varphi\zeta+\rho\psi\zeta)+\frac{p(\tilde{\rho})_{x}}{\tilde{\rho}}\varphi\psi\right]dx}_{J_{1}}
\\&\underbrace{-\alpha\beta\int_{\R_{+}}
W_{\alpha-1,\beta}\left(\lambda\psi\psi_{x}+\nu\zeta\zeta_{x}\right)dx}_{J_{2}},
\end{aligned}
\end{equation}
where we have used boundary condition $\psi(0,t)=0$ and $\zeta(0,t)=0$.

Now we estimate each term in $\eqref{z.eng1*}$.  We decompose $\rho$
as $\rho=\varphi+(\tilde{\rho}-\rho_{+})+\rho_{+},$ $u$ as
$u=\psi+(\tilde{u}-u_{+})+u_{+}$ and $\omega$ as
$\omega=\zeta+\tilde{\omega}$. Then we see,
under the condition $M_{+}>1$ and $u_{+}<0$, that
\begin{equation*}\label{z.eng1**}
\begin{aligned}[b]
&-\left[\rho u\Phi(\rho,\tilde{\rho})+\frac{\rho
u}{2}\psi^{2}+\frac{\rho u}{2}\zeta^{2}+(p(\rho)-p(\tilde{\rho}))\psi\right]\\
 \geq &\left[-\frac{\gamma Ku_{+}}{2}\rho_{+}^{\gamma-2}\varphi^{2}-K\gamma\rho_{+}^{\gamma-1}\varphi\psi-\frac{\rho_{+}u_{+}}{2}\psi^{2}
\right]-\frac{\rho_{+}u_{+}}{2}\zeta^{2}
\\&-C(N_{1}(T)+\tilde{\delta})(\varphi^{2}+\psi^{2}+\zeta^{2})\\
=&[\varphi,\psi]M_{1}[\varphi,\psi]^{T}
+\frac{\rho_{+}|u_{+}|}{2}\zeta^{2}
-C(N_{1}(T)+\tilde{\delta})(\varphi^{2}+\psi^{2}+\zeta^{2}),
\end{aligned}
\end{equation*}
 where $[\,]^T$ denotes the transpose of a row vector, and the
$2\times 2$ real symmetric matrix $M_{1}$ is given by
\begin{eqnarray*}
\begin{aligned} &\left(\begin{array} {cccccc}
-\frac{\gamma Ku_{+}}{2}\rho_{+}^{\gamma-2}\  \  \  & -\frac{K\gamma\rho_{+}^{\gamma-1}}{2} \\
 -\frac{K\gamma\rho_{+}^{\gamma-1}}{2} \  \  \ & -\frac{\rho_{+}u_{+}}{2} \\
\end{array} \right).
\end{aligned}
\end{eqnarray*}
One can compute all the leading principal minors  $\Delta_{ll}$
$(1\leq l\leq 2)$  of $M_{1}$ as follows:
\begin{equation*}
\begin{split}
\Delta_{11}=-\frac{\gamma Ku_{+}}{2}\rho_{+}^{\gamma-2}>0,\ \ \
\Delta_{22}=\frac{\gamma K}{4}\rho_{+}^{\gamma-1}(u_{+}^{2}-\gamma K\rho_{+}^{\gamma-1})>0,
\end{split}
\end{equation*}
where we have used the condition $M_{+}>1$ and $u_{+}<0.$

Thus we have
\begin{equation*}\label{J11fdkp}
\begin{aligned}[b]
-\alpha\beta\int_{\R_{+}}
W_{\alpha-1,\beta}\left[\rho u\Phi(\rho,\tilde{\rho})+\frac{\rho
u}{2}\psi^{2}+\frac{\rho u}{2}\zeta^{2}+(p(\rho)-p(\tilde{\rho}))\psi\right]dx \geq
c\beta\|[\varphi,\psi,\zeta]\|_{\alpha-1,\beta}^{2},
\end{aligned}
\end{equation*}
where we take $N_{1}(T)$ and $\tilde{\delta}$ small enough.

It is easy to obtain that
\begin{equation*}\label{Q95-2}
\begin{aligned}[b]
-\rho(0,t)u_{b}\Phi(\rho(0,t),\tilde{\rho}(0))\geq
c\varphi(0,t)^{2}.
\end{aligned}
\end{equation*}
Using  Theorem \ref{stationary}(i), Lemma \ref{lem.V}(i), \eqref{NSfvP.ptid} and
Cauchy-Schwarz's inequality with $0<\eta<1$, we have
\begin{equation*}\label{Q95-2}
\begin{aligned}[b]
|J_{1}| \leq C\tilde{\delta}\varphi(0,t)^{2}
+C\tilde{\delta}\|\partial_{x}[\varphi,\psi,\zeta]\|^{2},
\end{aligned}
\end{equation*}
\begin{equation*}\label{Q95-2}
\begin{aligned}[b]
|J_{2}|\leq \eta\|\partial_{x}[\psi,\zeta]\|_{\alpha-1,\beta}^{2} +C_{\eta}\beta^{2}\|[\psi,\zeta]\|_{\alpha-1,\beta}^{2}.
\end{aligned}
\end{equation*}

Inserting the above estimations into $\eqref{z.eng1*}$ and then
choosing $\eta$, $N_{1}(T)$, $\tilde{\delta}$ and $\beta$ suitably small,
we obtain
\begin{eqnarray}\label{3.44ew}
\begin{aligned}[b]
\frac{d}{dt}&\int_{\R_{+}}W_{\alpha,\beta}\left[\rho\Phi(\rho,\tilde{\rho})+\frac{\rho}{2}\psi^{2}+\frac{\rho}{2}\zeta^{2}\right] dx +c\varphi(0,t)^{2}
\\&+c\beta\|[\varphi,\psi,\zeta]\|_{\alpha-1,\beta}^{2}
+c\|[\zeta,\pa_{x}[\psi,\zeta]]\|_{\alpha,\beta}^{2} \leq
C\tilde{\delta}\|\partial_{x}\varphi\|^{2}.
\end{aligned}
\end{eqnarray}

Multiplying $\eqref{3.44ew}$ by $(1+t)^{\xi}$ and integrating in
$\tau$ over $[0,t]$ for any $0\leq t \leq T$, we have the
desired estimate \eqref{1m} for $\alpha \in (0,\vartheta]$.
Next, we prove \eqref{1m} holds for $\alpha=0$.
Multiplying \eqref{basic*} by  $(1+t)^{\xi}$ and integrating the resulting
equality  over $\R_{+}\times (0,t)$, then \eqref{1m} holds for $\alpha=0$
with the aid of Theorem \ref{stationary}(i), Lemma \ref{lem.V}(i) and boundary condition \eqref{NSfvP.ptid}.
\end{proof}

\begin{lemma}\label{lem.V1bnmv}
There exists a positive constant $\varepsilon_{1}$ such that if
$N_{1}(T)+\tilde{\delta}+\beta\leq\varepsilon_{1}$, then
\begin{eqnarray}\label{2m}
\begin{aligned}[b]
  &(1+
t)^{\xi}\| \pa_{x}\varphi\|^{2}+\int_{0}^{t}(1+
\tau)^{\xi}\|\partial_{x}\varphi\|^{2}d\tau+\int_{0}^{t}(1+
\tau)^{\xi}(\pa_x\varphi)^{2}(0,\tau)d\tau\\
 \leq &C \left(\|[\varphi_{0},\psi_{0},\zeta_{0}]\|_{\vartheta,\beta}^{2}
+\|\partial_{x}\varphi_{0}\|^{2}\right) +\xi\int_{0}^{t}(1+
\tau)^{\xi-1}\left(\|[\varphi,\psi,\zeta](\tau)\|_{\alpha,\beta}^{2}+\|\partial_{x}\varphi\|^{2}\right)d\tau
\\ &+CN_{1}(T)\int_{0}^{t}(1+\tau)^{\xi}\|\pa^{2}_x\psi\|^{2}d\tau
\end{aligned}
\end{eqnarray}
holds for $\xi\geq0$.
\end{lemma}

\begin{proof}
We first differentiate  $\eqref{NSP.pt}_{1}$  with respect to $x$,
multiplying the resulting equations and $\eqref{NSP.pt}_{2}$ by
 $\frac{\lambda\pa_x\varphi}{\rho^{2}}$ and $\frac{\pa_x\varphi}{\rho}$ respectively to obtain
\begin{equation}\label{d.rho.ip3}
\begin{aligned}[b]
&\lambda\frac{\pa_x\varphi}{\rho^{2}}\pa_t\pa_x\varphi
+\lambda\pa_xu\frac{(\pa_x\varphi)^{2}}{\rho^{2}}+\lambda
u\frac{\pa_x\varphi\pa^{2}_x\varphi}{\rho^{2}} +
\lambda\frac{\pa_x\varphi}{\rho^{2}}\pa_x\rho\pa_x\psi
+\lambda\pa_x\tilde{u}\frac{(\pa_x\varphi)^{2}}{\rho^{2}}\\
&\qquad =-\lambda\pa^2_x \psi\frac{\pa_x\varphi}{\rho}
-\lambda\pa^{2}_x\tilde{u}\varphi\frac{\pa_x\varphi}{\rho^{2}}-
\lambda\pa_x\tilde{\rho}\pa_x\psi\frac{\pa_x\varphi}{\rho^{2}}-\lambda\pa^{2}_x\tilde{\rho}\psi\frac{\pa_x\varphi}
{\rho^{2}},
\end{aligned}
\end{equation}

\begin{equation}\label{tu1.ip1}
\begin{aligned}[b]
&\pa_t\psi\pa_x\varphi +u\pa_x\psi \pa_x\varphi
 +\pa_x
[p(\rho)-p(\tilde{\rho})]\frac{\pa_x\varphi}{\rho}=\lambda\pa^2_x
\psi\frac{\pa_x\varphi}{\rho}
-\frac{\tilde{u}\pa_x\tilde{u}}{\rho}\varphi
\pa_x\varphi-\pa_x\tilde{u}\psi\pa_x\varphi.
\end{aligned}
\end{equation}

The summation of  \eqref{d.rho.ip3} and \eqref{tu1.ip1},
 and then taking integration
over $\R_{+}$ further imply
\begin{equation}\label{sum.d1}
\begin{aligned}[b]
\frac{d}{dt}&\int_{\R_{+}}\left[\psi \pa_x\varphi+\lambda
\frac{(\pa_x\varphi)^{2}}{2\rho^{2}}\right]dx
+\int_{\R_{+}}
\frac{p'(\tilde{\rho})}{\rho}(\pa_x\varphi)^{2}dx\\
=&\int_{\R_{+}}\psi \pa_t\pa_x\varphi dx
-\lambda\int_{\R_{+}}(\pa_x\varphi)^{2}\rho^{-3}\partial_{t}\rho
dx-\int_{\R_{+}}\pa_x
[p(\rho)-p(\tilde{\rho})-p'(\tilde{\rho})\varphi]\frac{\pa_x\varphi}{\rho}dx
\\&-\int_{\R_{+}}
\frac{\pa_xp'(\tilde{\rho})}{\rho}\varphi\pa_x\varphi dx
-\int_{\R_{+}}u\pa_x\psi \pa_x\varphi
dx-\int_{\R_{+}}\frac{\tilde{u}\pa_x\tilde{u}}{\rho}\varphi
\pa_x\varphi dx-\int_{\R_{+}}\pa_x\tilde{u}\psi\pa_x\varphi dx\\&
-\lambda\int_{\R_{+}}\pa_xu\frac{(\pa_x\varphi)^{2}}{\rho^{2}}dx
-\lambda\int_{\R_{+}}u\frac{\pa_x\varphi\pa^{2}_x\varphi}{\rho^{2}}
dx-\lambda\int_{\R_{+}} \frac{\pa_x\varphi}{\rho^{2}}\pa_x\rho\pa_x\psi
dx-\lambda\int_{\R_{+}}\pa^{2}_x\tilde{u}\varphi\frac{\pa_x\varphi}{\rho^{2}}
dx\\&-\lambda
\int_{\R_{+}}\pa_x\tilde{\rho}\pa_x\psi\frac{\pa_x\varphi}{\rho^{2}}
dx-\lambda\int_{\R_{+}}\pa^{2}_x\tilde{\rho}\psi\frac{\pa_x\varphi}{\rho^{2}}
dx-\int_{\R_{+}}\lambda\pa_x\tilde{u}\frac{(\pa_x\varphi)^{2}}{\rho^{2}}
dx=\sum\limits_{l=3}^{16}J_{l},
\end{aligned}
\end{equation}
where $J_l$ $(3\leq l\leq 15)$  denote the
corresponding terms on the left of \eqref{sum.d1}.

Applying Sobolev's inequality, Young's inequality and
Cauchy-Schwarz's inequality with $0<\eta<1$ and using Theorem \ref{stationary}(i), Lemma
\ref{lem.V}(i),  one has
\begin{equation*}
\begin{aligned}[b]
J_3=&\int_{\R_{+}}\pa_x\psi \pa_x(\rho u-\tilde{\rho}
\tilde{u})dx\\
=&\int_{\R_{+}}\rho(\pa_x\psi)^{2}dx+\int_{\R_{+}}\pa_x\tilde{\rho}
\psi\pa_x\psi dx+\int_{\R_{+}}\varphi\pa_x\tilde{u}\pa_x\psi
dx+\int_{\R_{+}}u\pa_x\psi\pa_x\varphi dx\\
\leq&(\eta+C\tilde{\delta})\|\pa_x\varphi\|^{2}+(C_{\eta}+C\tilde{\delta})\|\pa_x\psi\|^{2}
+C\tilde{\delta}\varphi(0,t)^{2},
\end{aligned}
\end{equation*}
\begin{equation*}
\begin{aligned}[b]
&J_4+J_{10}+J_{11}\\=&
-\frac{\lambda|u_{b}|}{2\rho(0,t)^{2}}(\pa_x\varphi)^{2}(0,t)
+\frac{\lambda}{2}\int_{\R_{+}}\pa_x
\tilde{u}(\pa_x\varphi)^{2}\rho^{-2}dx+\frac{\lambda}{2}\int_{\R_{+}}\pa_x
\psi(\pa_x\varphi)^{2}\rho^{-2}dx\\
\leq&-\frac{\lambda|u_{b}|}{2\rho(0,t)^{2}}(\pa_x\varphi)^{2}(0,t)
+C(N_{1}(T)+\tilde{\delta})\|\pa_x\varphi\|^{2}+CN_{1}(T)\|\pa^{2}_x\psi\|^{2},
\end{aligned}
\end{equation*}
\begin{equation*}
\begin{aligned}[b]
&|J_{7}|+|J_{12}|+|J_{14}|+|J_{16}|\\ \leq &
(\eta+CN_{1}(T)+C\tilde{\delta})\|\pa_x\varphi\|^2
+(C_{\eta}+C\tilde{\delta})\|\pa_x\psi\|^2
+CN_{1}(T)\|\pa^{2}_x\psi\|^{2},
\end{aligned}
\end{equation*}
\begin{equation*}
\begin{aligned}[b]
&|J_{5}|+|J_{6}|+|J_{8}|+|J_{9}|+|J_{13}|+|J_{15}| \\ \leq&
C(N_{1}(T)+\tilde{\delta})\|\pa_x[\varphi,\psi]\|^2
+C\tilde{\delta}\varphi^{2}(0,t).
\end{aligned}
\end{equation*}

Inserting the above estimates for $J_l$ $(3\leq l\leq 15)$ into
\eqref{sum.d1} and then choosing $\eta$, $N_{1}(T)$ and $\tilde{\delta}$
 suitably small, we obtain
\begin{equation}\label{1.eng1}
\begin{aligned}[b]
&\frac{d}{dt}\int_{\R_{+}}\left[\psi
\pa_x\varphi+\lambda\frac{(\pa_x\varphi)^{2}}{2\rho^{2}} \right]dx
+\|\pa_{x}\varphi\|^{2}
+(\pa_x\varphi)^{2}(0,t) \\
\leq & C\|\pa_x\psi\|^2
+CN_{1}(T)\|\pa^{2}_x\psi\|^{2}+\varphi^{2}(0,t).
\end{aligned}
\end{equation}

Multiplying $\eqref{1.eng1}$ by $(1+t)^{\xi}$ and integrating in
$\tau$ over $[0,t]$ for any $0\leq t \leq T$, using \eqref{1m} and
Cauchy-Schwarz's inequality, one has \eqref{2m}.
\end{proof}

\begin{lemma}\label{lem.V1bnmvhj}
There exists a positive constant $\varepsilon_{1}$ such that if
$N_{1}(T)+\tilde{\delta}+\beta\leq\varepsilon_{1}$, then
\begin{eqnarray}\label{3m}
\begin{aligned}[b]
  &(1+
t)^{\xi}\| \pa_{x}[\psi,\zeta](t)\|^{2}+\int_{0}^{t}(1+
\tau)^{\xi}\|\partial_{x}^{2}[\psi,\zeta]\|^{2}d\tau\\
 \leq &C \left(\|[\varphi_{0},\psi_{0},\zeta_{0}]\|_{\vartheta,\beta}^{2}
+\|\partial_{x}[\varphi_{0},\psi_{0},\zeta_{0}]\|^{2}\right)\\
&
 +\xi\int_{0}^{t}(1+
\tau)^{\xi-1}\left(\|[\varphi,\psi,\zeta](\tau)\|_{\alpha,\beta}^{2}
+\|\partial_{x}[\varphi,\psi,\zeta]\|^{2}\right)d\tau
\end{aligned}
\end{eqnarray}
holds for $\xi\geq0$.
\end{lemma}

\begin{proof}
Multiplying $\eqref{NSP.pt}_{2}$ by $-\frac{\pa^{2}_x\psi}{\rho}$,
and then integrating the resulting equations over $\R_{+}$, one has
\begin{equation}\label{2.eng1}
\begin{aligned}[b]
&\frac{d}{dt}\int_{\R_{+}}\frac{(\pa_x\psi)^{2}}{2}dx
+\lambda\int_{\R_{+}}\frac{(\pa^{2}_x\psi)^{2}}{\rho}dx
\\
\qquad=&\underbrace{\int_{\R_{+}}\frac{\pa_x
[p(\rho)-p(\tilde{\rho})]}{\rho}\pa^{2}_x\psi
dx}_{J_{17}} \underbrace{+\int_{\R_{+}}u\pa_x\psi\pa^{2}_x\psi
dx}_{J_{18}}\underbrace{+\int_{\R_{+}}\frac{\tilde{u}\pa_x\tilde{u}}{\rho}\varphi\pa^{2}_x\psi
dx}_{J_{19}}\underbrace{+\int_{\R_{+}}\pa_x\tilde{u}\psi\pa^{2}_x\psi
dx}_{J_{20}}.
\end{aligned}
\end{equation}
We utilize integration by parts, Cauchy-Schwarz's inequality and
Lemma \ref{lem.V} to address the following estimates:
\begin{equation*}
\begin{split}
|J_{18}|\leq\eta\|\pa^{2}_x\psi\|^{2}+C_{\eta}\|\pa_x[\varphi,\psi]\|^{2},
\end{split}
\end{equation*}
and
\begin{equation*}
\begin{split}
|J_{17}|+|J_{19}|+|J_{20}|\leq\eta\|\pa^{2}_x\psi\|^{2}+C_{\eta}\tilde{\delta}\|\partial_{x}[\varphi,\psi]\|^2
+C\tilde{\delta}\varphi^{2}(0,t).
\end{split}
\end{equation*}
Substituting the above estimates for $J_{l}$ $(17\leq l\leq 20)$
into \eqref{2.eng1} and taking $\eta$  small enough, one has
\begin{equation}\label{2.eng2}
\begin{aligned}[b]
&\frac{d}{dt}\int_{\R_{+}}\frac{(\pa_x\psi)^{2}}{2}dx+\|\partial_{x}^{2}\psi\|^{2}
\leq C\|\partial_{x}[\varphi,\psi]\|^2 +C\varphi^{2}(0,t).
\end{aligned}
\end{equation}

Multiplying $\eqref{NSP.pt}_{3}$ by $-\frac{\pa^{2}_x\zeta}{\rho}$,
and integrating the resulting equality over $\R_{+}$, we obtain
\begin{equation}\label{sum.ip5}
\begin{split}
\frac{1}{2}&\frac{d}{dt}\int_{\R_{+}}(\pa_x\zeta)^{2}dx+\nu\int_{\R_{+}}\frac{(\pa^{2}_x\zeta)^{2}}{\rho}dx
=\underbrace{\int_{\R_{+}}u\pa_x\zeta\pa^{2}_x\zeta
dx}_{J_{21}}\underbrace{+\mu\int_{\R_{+}}\frac{\zeta}{\rho}\pa^{2}_x\zeta
dx}_{J_{22}}\underbrace{+\int_{\R_{+}}\frac{\pa_x\tilde{\omega}}{\rho}(\tilde{u}\varphi+\rho\psi)\pa^{2}_x\zeta dx}_{J_{23}}.
\end{split}
\end{equation}
To obtain the estimates for $J_{21}$-$J_{23}$, we use
Cauchy-Schwarz's inequality with $0<\eta<1$ to get
\begin{equation*}\label{3.3}
\begin{split}
|J_{21}|+|J_{22}| \leq
\eta\|\pa^{2}_x\zeta\|^{2}+C_{\eta}\|[\zeta,\pa_x\zeta]\|^{2},
\end{split}
\end{equation*}
\begin{equation*}
\begin{split}
|J_{23}|\leq\eta\|\pa^{2}_x\zeta\|^{2}+C_{\eta}\tilde{\delta}\|\partial_{x}[\varphi,\psi]\|^2
+C\tilde{\delta}\varphi^{2}(0,t).
\end{split}
\end{equation*}

Then we have
\begin{equation}\label{2.eng2asx}
\begin{aligned}[b]
&\frac{d}{dt}\int_{\R_{+}}\frac{(\pa_x\zeta)^{2}}{2}dx+\|\partial_{x}^{2}\zeta\|^{2}
\leq C\|[\zeta,\pa_x\zeta]\|^{2}+C\tilde{\delta}\|\partial_{x}[\varphi,\psi]\|^2
+C\tilde{\delta}\varphi^{2}(0,t)
\end{aligned}
\end{equation}
if $\eta$ is small enough.

The summation of $\eqref{2.eng2}$ and $\eqref{2.eng2asx}$, and multiplying the resulting inequality
 by $(1+t)^{\xi}$, then integrating the resulting inequality in $\tau$
over $[0,t]$ for any $0\leq t \leq T$, using \eqref{1m}, \eqref{2m}
and Cauchy-Schwarz's inequality, one has \eqref{3m}.
\end{proof}

\vspace{5mm}

\textbf{Proof of Proposition \ref{priori.est}}\ \
 Now, following the three steps above, we are ready to
prove Proposition
 $\ref{priori.est}$.  Summing up the estimates $\eqref{1m}$,
$\eqref{2m}$ and $\eqref{3m}$, and taking $\tilde{\delta}$ and
$N_{1}(T)$ suitably small, we have
\begin{eqnarray}\label{4m}
\begin{aligned}[b]
&(1+
t)^{\xi}\left(\|[\varphi,\psi,\zeta](t)\|_{\alpha,\beta}^{2}+\|
\pa_{x}[\varphi,\psi,\zeta](t)\|^{2}\right)+\int_{0}^{t}(1+
\tau)^{\xi}\|\partial_{x}[\varphi,\partial_{x}[\psi,\zeta]](\tau)\|^{2}d\tau\\&
+\int_{0}^{t}(1+ \tau)^{\xi}
\left(\beta\|[\varphi,\psi,\zeta,](\tau)\|_{\alpha-1,\beta}^{2}
+\|[\zeta,\partial_{x}[\psi,\zeta]](\tau)\|_{\alpha,\beta}^{2}\right)d\tau\\
\leq &C
\left(\|[\varphi_{0},\psi_{0},\zeta_{0}]\|_{\vartheta,\beta}^{2}
+\|\partial_{x}[\varphi_{0},\psi_{0},\zeta_{0}]\|^{2}\right)\\
&+\xi\int_{0}^{t}(1+
\tau)^{\xi-1}\left(\|[\varphi,\psi,\zeta](\tau)\|_{\alpha,\beta}^{2}
+\|\partial_{x}[\varphi,\psi,\zeta]\|^{2}\right)d\tau,
\end{aligned}
\end{eqnarray}
where $C$ is a positive constant independent of $T$, $\alpha$, $\beta$,
$N_{1}(T)$ and $\tilde{\delta}$. Hence, similarly as in
\cite{S. Kawashima, M. Nishikawasll}, applying an induction to
$\eqref{4m}$ gives desired estimate $\eqref{3.3x}.$

\subsection{The a priori estimates for $M_{+}=1$}
This subsection is devoted to prove the algebraic decay estimate for the transonic case $M_{+}=1$
in Theorem $\mathrm{\ref{1.2theorem}}$. For the convenience of stating the a priori assumption, we use the notations
\begin{eqnarray}\label{3.2nb}
 \begin{aligned}[b]
N_{2}(T):=\sup\limits_{0\leq t\leq
T}\|(\varphi,\psi,\zeta)(t)\|_{\vartheta,B,1}.
\end{aligned}
\end{eqnarray}

\begin{proposition}\label{priori.estfvb}
 Assume  the same conditions as in Theorem $\mathrm{\ref{1.2theorem}}$ (ii) hold.
 Suppose $[\varphi,\psi,\zeta]\in X([0,T))$
  is a solution to
\eqref{NSP.pt}, \eqref{NSP.ptid} and \eqref{NSfvP.ptid} which
satisfies $(1+ B
x)^{\frac{\vartheta}{2}}\varphi,(1+B x)^{\frac{\vartheta}{2}}\psi,(1+B
x)^{\frac{\vartheta}{2}}\zeta\in C([0,T];H^1(\mathbb{R}_{+}))$ for certain
positive constants $T$, $B$ and $\vartheta\in[2,\vartheta^{*})$, where $B$ and $\vartheta^{*}$ is respectively defined in \eqref{stationary1} and \eqref{1.thre1b}. For arbitrary $\alpha\in[0,\vartheta]$, there
exist positive constants $C$ and $\varepsilon_{2}$ independent of
$T$ such that if $N_{2}(T)+\tilde{\delta}\leq
\varepsilon_{2}$ is satisfied, it holds for an arbitrary $t\in[0,T]$ that
\begin{eqnarray}\label{3.3mnh}
 \begin{aligned}[b]
&(1+t)^{\frac{\alpha}{2}+\kappa}\|[\varphi,\psi,\zeta](t)\|_{\alpha,B,1}^{2}
+\int_{0}^{t}(1+
\tau)^{\frac{\alpha}{2}+\kappa}\|[\zeta,\partial_{x}[\varphi,\psi,\zeta],\partial_{x}^{2}[\psi,\zeta]](\tau)\|_{\alpha,B}^{2}d\tau
\\
\leq & C (1+ t)^{\kappa}\|[\varphi_{0},
\psi_{0},\zeta_{0}]\|_{\vartheta,B,1}^{2},
\end{aligned}
\end{eqnarray}
where $\kappa$ is positive constant.
\end{proposition}

In order to prove Proposition \ref{priori.estfvb}, we need to get a lower estimate for $\tilde{u}(x)$ and the Mach number $\tilde{M}$ on the stationary solution $(\tilde{\rho}(x),\tilde{u}(x))$  defined
by $\tilde{M}:=\frac{|\tilde{u}(x)|}{\sqrt{p'(\tilde{\rho}(x))}}$.

\begin{lemma}\label{lem.Vavbc}(See \cite{KNZhd})
The stationary solution $\tilde{u}(x)$ satisfies
\begin{equation}\label{stationary1}
\begin{split}
\tilde{u}_{x}(x)\geq A\left(\frac{u_{+}}{u_{b}}\right)^{\gamma+2}\frac{\tilde{\delta}^{2}}{(1+Bx)^{2}}>0,\ \ \
 A:=\frac{(\gamma+1)\rho_{+}}{2\lambda},\ \ \ B=\tilde{\delta} A
\end{split}
\end{equation}
for $x\in(0,\infty)$. Moreover, there exists a positive constant $C$ such that
\begin{equation}\label{stationary2}
\begin{split}
\frac{(\gamma+1)\tilde{\delta}}{2|u_{+}|(1+Bx)}-C\frac{\tilde{\delta}^{2}}{(1+Bx)^{2}}
\leq\tilde{M}-1\leq C\frac{\tilde{\delta}}{1+Bx}.
\end{split}
\end{equation}
\end{lemma}
\begin{proof}
Please refer to \cite{KNZhd} for the detailed proof.
\end{proof}

\begin{lemma}\label{lem.V1}
There exists a positive constant $\varepsilon_{2}$ such that if
$N_{2}(T)+\tilde{\delta}\leq\varepsilon_{2}$, then
\begin{eqnarray}\label{1mnb}
\begin{aligned}[b]
  &(1+
t)^{\xi}\|[\varphi,\psi,\zeta](t)\|_{\alpha,B}^{2}+\tilde{\delta}^{2}\int_{0}^{t}(1+
\tau)^{\xi}
\|[\varphi,\psi](\tau)\|_{\alpha-2,B}^{2}d\tau+\tilde{\delta}\int_{0}^{t}(1+
\tau)^{\xi}
\|\zeta(\tau)\|_{\alpha-1,B}^{2}d\tau\\&+\int_{0}^{t}(1+
\tau)^{\xi}\|[\zeta,\partial_{x}[\psi,\zeta]](\tau)\|_{\alpha,B}^{2}d\tau+\int_{0}^{t}(1+
\tau)^{\xi}\varphi^{2}(0,\tau)d\tau\\
 \leq &C\|[\varphi_{0},\psi_{0},\zeta_{0}]\|_{\vartheta,B}^{2} +\xi\int_{0}^{t}(1+
\tau)^{\xi-1}\|[\varphi,\psi,\zeta](\tau)\|_{\alpha,B}^{2}d\tau
+C\tilde{\delta}\int_{0}^{t}(1+
\tau)^{\xi}\|\partial_{x}\varphi\|_{\alpha,B}^{2}(\tau)d\tau
\end{aligned}
\end{eqnarray}
holds for $\alpha \in [0,\vartheta]$ and $\xi\geq0$.
\end{lemma}
\begin{proof}
Using $\eqref{1.6}_{1}$, the equation \eqref{basic*} is rewritten to
\begin{eqnarray}\label{basic*bv}
\begin{aligned}[b]
&\left[\rho\Phi(\rho,\tilde{\rho})+\frac{\rho}{2}\psi^{2}+\frac{\rho}{2}\zeta^{2}\right]_{t}
+\left[\rho u\Phi(\rho,\tilde{\rho})+\frac{\rho
u}{2}\psi^{2}+\frac{\rho u}{2}\zeta^{2}+(p(\rho)-p(\tilde{\rho}))\psi-\lambda\psi\psi_{x}-\nu\zeta\zeta_{x}\right]_{x}\\&+\mu\zeta^{2}+\lambda\psi_{x}^{2}
+\nu\zeta_{x}^{2}+\pa_x\tilde{u}[\rho\psi^{2}+p(\rho)-p(\tilde{\rho})-p'(\tilde{\rho})\varphi]
=-\pa_x\tilde{\omega}(\tilde{u}\varphi\zeta+\rho\psi\zeta)-\frac{\lambda\pa_x^{2}\tilde{u}}{\tilde{\rho}}\varphi\psi.
\end{aligned}
\end{eqnarray}
Multiplying \eqref{basic*bv} by $W_{\alpha,B}:=(1+Bx)^{\alpha}$ defined in Notation,
then we integrate the resulting equality over $\mathbb{R}_{+}$ to
get
\begin{equation}\label{z.eng1*12}
\begin{aligned}[b]
\frac{d}{dt}&\int_{\R_{+}}W_{\alpha,B}\left[\rho\Phi(\rho,\tilde{\rho})+\frac{\rho}{2}\psi^{2}+\frac{\rho}{2}\zeta^{2}\right]dx
+\int_{\R_{+}}W_{\alpha,B}\left[\mu\zeta^{2}+\lambda\psi_{x}^{2}
+\nu\zeta_{x}^{2}\right]dx\\&
-\rho(0,t)u_{b}\Phi(\rho(0,t),\tilde{\rho}(0))+\underbrace{\alpha B\int_{\R_{+}}
W_{\alpha-1,B}\left[-\rho u\Phi(\rho,\tilde{\rho})-\frac{\rho
u}{2}\psi^{2}-(p(\rho)-p(\tilde{\rho}))\psi\right]}_{G_{1}}dx\\&
\underbrace{+\int_{\R_{+}}W_{\alpha,B}\pa_x\tilde{u}\left[\rho\psi^{2}+p(\rho)-p(\tilde{\rho})-p'(\tilde{\rho})\varphi\right]
dx}_{G_{2}}\underbrace{-\frac{\lambda\alpha(\alpha-1)}{2} B^{2}\int_{\R_{+}}
W_{\alpha-2,B}\psi^{2}dx}_{G_{3}}\\&
\underbrace{-\frac{\alpha B}{2}\int_{\R_{+}}
W_{\alpha-1,B}\rho u\zeta^{2}dx}_{G_{4}}
\underbrace{-\frac{\nu\alpha(\alpha-1)}{2} B^{2}\int_{\R_{+}}
W_{\alpha-2,B}\zeta^{2}dx}_{G_{5}}\\
=&\underbrace{-\int_{\R_{+}}W_{\alpha,B}\pa_x\tilde{\omega}(\tilde{u}\varphi\zeta+\rho\psi\zeta)dx}_{K_{1}}
-\underbrace{\int_{\R_{+}}W_{\alpha,B}\frac{\lambda\pa_x^{2}\tilde{u}}{\tilde{\rho}}\varphi\psi dx}_{K_{2}},
\end{aligned}
\end{equation}
where we have used integration by parts and boundary condition \eqref{NSfvP.ptid}.

It is easy to obtain that
\begin{equation*}\label{Q95-2}
\begin{aligned}[b]
-\rho(0,t)u_{b}\Phi(\rho(0,t),\tilde{\rho}(0))\geq
c\varphi(0,t)^{2}.
\end{aligned}
\end{equation*}
Using Lemma \ref{lem.Vavbc} and  decomposing $\rho$
as $\rho=\varphi+(\tilde{\rho}-\rho_{+})+\rho_{+},$ $u$ as
$u=\psi+(\tilde{u}-u_{+})+u_{+}$, we have the following estimates:
\begin{equation}\label{z.eng1*1}
\begin{aligned}[b]
G_{1}\geq& \alpha B\int_{\R_{+}}
W_{\alpha-1,B}\left[\left(\frac{p'(\rho_{+})^{\frac{3}{2}}}{2\rho_{+}}\varphi^{2}
+\frac{\rho_{+}\sqrt{p'(\rho_{+})}}{2}\psi^{2}\right)(\tilde{M}-1)
+\frac{p'(\tilde{\rho})}{2\tilde{\rho}}(\sqrt{p'(\tilde{\rho})}\varphi-\tilde{\rho}\psi)^{2}\right]dx\\
&-C\alpha B\|(1+Bx)[\varphi,\psi]\|_{L^{\infty}}\|[\varphi,\psi]\|_{\alpha-2,B}^{2}-C\tilde{\delta}\alpha B\int_{\R_{+}}
W_{\alpha-1,B}(\varphi^{2}+\psi^{2})(\tilde{M}-1)dx\\
\geq&\alpha B\int_{\R_{+}}
W_{\alpha-1,B}\left(\frac{p'(\rho_{+})^{\frac{3}{2}}}{2\rho_{+}}\varphi^{2}
+\frac{\rho_{+}\sqrt{p'(\rho_{+})}}{2}\psi^{2}\right)\left(\frac{(\gamma+1)\tilde{\delta}}{2|u_{+}|(1+Bx)}-C\frac{\tilde{\delta}^{2}}{(1+Bx)^{2}}\right)
dx\\
&-C\alpha B\|(1+Bx)[\varphi,\psi]\|_{L^{\infty}}\|[\varphi,\psi]\|_{\alpha-2,B}^{2}-C\tilde{\delta}^{2}\alpha B\|[\varphi,\psi]\|_{\alpha-2,B}^{2}\\
\geq&\frac{(\gamma+1)\alpha A}{2|u_{+}|}\tilde{\delta}^{2}\int_{\R_{+}}
W_{\alpha-2,B}\left(\frac{p'(\rho_{+})^{\frac{3}{2}}}{2\rho_{+}}\varphi^{2}
+\frac{\rho_{+}\sqrt{p'(\rho_{+})}}{2}\psi^{2}\right)
dx\\&-C\tilde{\delta}(N_{2}(T)+\tilde{\delta}^{2})\|[\varphi,\psi]\|_{\alpha-2,B}^{2},
\end{aligned}
\end{equation}

\begin{equation}\label{z.eng1bv*}
\begin{aligned}[b]
G_{2}\geq&
A\left(\frac{u_{+}}{u_{b}}\right)^{\gamma+2}\tilde{\delta}^{2}\int_{\R_{+}}W_{\alpha-2,B}\left[\rho\psi^{2}+p(\rho)-p(\tilde{\rho})-p'(\tilde{\rho})\varphi\right]
dx\\
\geq&
A\left(\frac{u_{+}}{u_{b}}\right)^{\gamma+2}\tilde{\delta}^{2}\int_{\R_{+}}W_{\alpha-2,B}\left(\rho_{+}\psi^{2}+\frac{p''(\rho_{+})}{2}\varphi^{2}\right)
dx-C\tilde{\delta}^{2}(\|\varphi\|_{L^{\infty}}+\tilde{\delta})\|[\varphi,\psi]\|_{\alpha-2,B}^{2}
\\ \geq&
A\left(\frac{u_{+}}{u_{b}}\right)^{\gamma+2}\tilde{\delta}^{2}\int_{\R_{+}}W_{\alpha-2,B}\left(\rho_{+}\psi^{2}+\frac{p''(\rho_{+})}{2}\varphi^{2}\right)
dx-C\tilde{\delta}^{2}(N_{2}(T)+\tilde{\delta})\|[\varphi,\psi]\|_{\alpha-2,B}^{2}.
\end{aligned}
\end{equation}
For $\alpha\in(0,\vartheta]$, we obtain the lower estimate of $G_{1}+G_{2}+G_{3}$ as
\begin{equation}\label{z.eng1mbvnj*}
\begin{aligned}[b]
&G_{1}+G_{2}+G_{3}\\ \geq&
A\tilde{\delta}^{2}\int_{\R_{+}}W_{\alpha-2,B}\left[\frac{(\gamma+1)\alpha \rho_{+}\sqrt{p'(\rho_{+})}}{4|u_{+}|}+\rho_{+}\left(\frac{u_{+}}{u_{b}}\right)^{\gamma+2}-\frac{\lambda\alpha(\alpha-1)}{2} A\right]\psi^{2}
dx\\&+A\tilde{\delta}^{2}\int_{\R_{+}}W_{\alpha-2,B}\left[\frac{(\gamma+1)\alpha p'(\rho_{+})^{\frac{3}{2}} }{4\rho_{+}|u_{+}|}+\frac{p''(\rho_{+})}{2}\left(\frac{u_{+}}{u_{b}}\right)^{\gamma+2}\right]\varphi^{2}
dx\\&-C\tilde{\delta}(N_{2}(T)+\tilde{\delta}N_{2}(T)+\tilde{\delta}^{2})\|[\varphi,\psi]\|_{\alpha-2,B}^{2}\\
=&A\rho_{+}\tilde{\delta}^{2}\int_{\R_{+}}W_{\alpha-2,B}\left[\left(\frac{u_{+}}{u_{b}}\right)^{\gamma+2}-\frac{\alpha(\alpha-2)(\gamma+1)}{4}\right]\psi^{2}
dx\\&+A\tilde{\delta}^{2}\int_{\R_{+}}W_{\alpha-2,B}\left[\frac{(\gamma+1)\alpha p'(\rho_{+})^{\frac{3}{2}} }{4\rho_{+}|u_{+}|}+\frac{p''(\rho_{+})}{2}\left(\frac{u_{+}}{u_{b}}\right)^{\gamma+2}\right]\varphi^{2}
dx\\&-C\tilde{\delta}(N_{2}(T)+\tilde{\delta}N_{2}(T)+\tilde{\delta}^{2})\|[\varphi,\psi]\|_{\alpha-2,B}^{2}\\
\geq&\frac{A\rho_{+}(\gamma+1)}{4}\tilde{\delta}^{2}\int_{\R_{+}}W_{\alpha-2,B}\left[\frac{4}{\gamma+1}-\alpha(\alpha-2)\right]\psi^{2}
dx\\&+A\tilde{\delta}^{2}\int_{\R_{+}}W_{\alpha-2,B}\left[\frac{(\gamma+1)\alpha p'(\rho_{+})^{\frac{3}{2}} }{4\rho_{+}|u_{+}|}+\frac{p''(\rho_{+})}{2}\left(\frac{u_{+}}{u_{b}}\right)^{\gamma+2}\right]\varphi^{2}
dx\\&-C\tilde{\delta}(N_{2}(T)+\tilde{\delta}N_{2}(T)+\tilde{\delta}^{2})\|[\varphi,\psi]\|_{\alpha-2,B}^{2}
\\
\geq &c\tilde{\delta}^{2}\|[\varphi,\psi]\|_{\alpha-2,B}^{2}-C\tilde{\delta}(N_{2}(T)+\tilde{\delta}N_{2}(T)+\tilde{\delta}^{2})\|[\varphi,\psi]\|_{\alpha-2,B}^{2},
\end{aligned}
\end{equation}
where we have used  Lemma \ref{lem.Vavbc}, $\alpha<\vartheta^{*}$ and $M_{+}=1$.

Similarly, we have
\begin{equation}\label{z.eng1mbvnnbj*}
\begin{aligned}[b]
&G_{4}+G_{5}\\ \geq&
\frac{\alpha A\rho_{+} |u_{+}|}{2}\tilde{\delta}\|\zeta\|_{\alpha-1,B}^{2}
-\frac{\nu\alpha(\alpha-1)A^{2}}{2} \tilde{\delta}^{2}\|\zeta\|_{\alpha-2,B}^{2}
-C\tilde{\delta}(N_{2}(T)+\tilde{\delta})\|\zeta\|_{\alpha-1,B}^{2}\\
\geq&
\frac{\alpha A\rho_{+} |u_{+}|}{2}\tilde{\delta}\|\zeta\|_{\alpha-1,B}^{2}
-C\tilde{\delta}(N_{2}(T)+\tilde{\delta})\|\zeta\|_{\alpha-1,B}^{2}.
\end{aligned}
\end{equation}

Using Poincar\'{e}
type inequality \eqref{p.ine} and \eqref{inequality21g6}, we have
\begin{equation}\label{z.eng1mbvnnbj*}
\begin{aligned}[b]
K_{1}\leq& C\tilde{\delta}\int_{\R_{+}}W_{\alpha,B}e^{-\sigma x}(\varphi^{2}+\psi^{2}+\zeta^{2})dx\\
\leq& C\tilde{\delta}\int_{\R_{+}}W_{\alpha,B}e^{-\sigma x}\left[\varphi(0,t)^{2}+x\|\partial_{x}[\varphi,\psi,\zeta]\|^{2}\right]dx\\
\leq& C\tilde{\delta}\varphi(0,t)^{2}+C\tilde{\delta}\|\pa_x[\varphi,\psi,\zeta]\|^{2}
\leq C\tilde{\delta}\varphi(0,t)^{2}+C\tilde{\delta}\|\pa_x[\varphi,\psi,\zeta]\|_{\alpha,B}^{2},
\end{aligned}
\end{equation}
\begin{equation}\label{z.enbgg1mbvnnbj*}
\begin{aligned}[b]
K_{2}\leq C\tilde{\delta}^{3}\|[\varphi,\psi]\|_{\alpha-2,B}^{2}.
\end{aligned}
\end{equation}

Inserting the above estimations into $\eqref{z.eng1*12}$ and then
taking $N_{2}(T)$ and $\tilde{\delta}$ suitably small to satisfy $N_{2}(T)\ll\tilde{\delta}^{2}$ and
$\tilde{\delta}\ll1$, we obtain
\begin{eqnarray}\label{3.44ewcv}
\begin{aligned}[b]
\frac{d}{dt}&\int_{\R_{+}}W_{\alpha,B}\left[\rho\Phi(\rho,\tilde{\rho})+\frac{\rho}{2}\psi^{2}+\frac{\rho}{2}\zeta^{2}\right] dx +c\varphi(0,t)^{2}
+c\tilde{\delta}^{2}\|[\varphi,\psi]\|_{\alpha-2,B}^{2}\\&
+c\tilde{\delta}\|\zeta\|_{\alpha-1,B}^{2}+c\|[\zeta,\pa_{x}[\psi,\zeta]]\|_{\alpha,B}^{2} \leq
C\tilde{\delta}\|\partial_{x}\varphi\|_{\alpha,B}^{2}.
\end{aligned}
\end{eqnarray}
Multiplying $\eqref{3.44ewcv}$ by $(1+t)^{\xi}$ and integrating in
$\tau$ over $[0,t]$ for any $0\leq t \leq T$, we have the
desired estimate \eqref{1mnb} for $\alpha \in (0,\vartheta]$.
Next, we prove \eqref{1mnb} holds for $\alpha=0$.
Multiplying \eqref{basic*bv} by  $(1+t)^{\xi}$ and integrating the resulting
equality  over $\R_{+}\times (0,t)$, we have
\begin{eqnarray}\label{1mvbcnb}
\begin{aligned}[b]
  &(1+
t)^{\xi}\|[\varphi,\psi,\zeta](t)\|^{2}+\int_{0}^{t}(1+
\tau)^{\xi}\|[\zeta,\partial_{x}[\psi,\zeta]](\tau)\|^{2}d\tau+\int_{0}^{t}(1+
\tau)^{\xi}\varphi^{2}(0,\tau)d\tau\\
 \leq &C\|[\varphi_{0},\psi_{0},\zeta_{0}]\|^{2} +\xi\int_{0}^{t}(1+
\tau)^{\xi-1}\|[\varphi,\psi,\zeta](\tau)\|^{2}d\tau
\\&+C\int_{0}^{t}(1+
\tau)^{\xi}\int_{\R_{+}}|\pa_x\tilde{\omega}|(|\varphi\zeta|+|\psi\zeta|)dxd\tau
+C\int_{0}^{t}(1+
\tau)^{\xi}\int_{\R_{+}}|\pa_x^{2}\tilde{u}||\varphi\psi| dxd\tau,
\end{aligned}
\end{eqnarray}
where we have used the fact that $\pa_x\tilde{u}[\rho\psi^{2}+p(\rho)-p(\tilde{\rho})-p'(\tilde{\rho})\varphi]\geq 0 $ holds.
Applying Lemma
\ref{lem.V} to
the third term and fourth term on the right-hand side of \eqref{1mvbcnb} with the aid of \eqref{inequality21g6}, we obtain the estimate \eqref{1mnb} for the case of $\alpha=0$.
\end{proof}

\begin{lemma}\label{lem.V12}
There exists a positive constant $\varepsilon_{2}$ such that if
$N_{2}(T)+\tilde{\delta}\leq\varepsilon_{2}$, then
\begin{eqnarray}\label{2mvbc}
\begin{aligned}[b]
  &(1+
t)^{\xi}\| \pa_{x}\varphi\|_{\alpha,B}^{2}+\int_{0}^{t}(1+
\tau)^{\xi}\|\partial_{x}\varphi\|_{\alpha,B}^{2}d\tau+\int_{0}^{t}(1+
\tau)^{\xi}(\pa_x\varphi)^{2}(0,\tau)d\tau\\
 \leq &C \left(\|[\varphi_{0},\psi_{0},\zeta_{0}]\|_{\vartheta,B}^{2}
+\|\partial_{x}\varphi_{0}\|_{\vartheta,B}^{2}\right) +\xi\int_{0}^{t}(1+
\tau)^{\xi-1}\left(\|[\varphi,\psi,\zeta](\tau)\|_{\alpha,B}^{2}+\|\partial_{x}\varphi\|_{\alpha,B}^{2}\right)d\tau
\\ &+CN_{2}(T)\int_{0}^{t}(1+\tau)^{\xi}\|\pa^{2}_x\psi\|_{\alpha,B}^{2}d\tau
\end{aligned}
\end{eqnarray}
holds for $\alpha \in [0,\vartheta]$ and $\xi\geq0$.
\end{lemma}
\begin{proof}
Multiplying the summation of  \eqref{d.rho.ip3} and \eqref{tu1.ip1} by $W_{\alpha,B}$, and then taking integration over $\R_{+}$ further imply
\begin{equation}\label{sum.d1vcf}
\begin{aligned}[b]
\frac{d}{dt}&\int_{\R_{+}}W_{\alpha,B}\left[\psi \pa_x\varphi+\lambda
\frac{(\pa_x\varphi)^{2}}{2\rho^{2}}\right]dx
+\int_{\R_{+}}W_{\alpha,B}
\frac{p'(\tilde{\rho})}{\rho}(\pa_x\varphi)^{2}dx\\
=&\int_{\R_{+}}W_{\alpha,B}\psi \pa_t\pa_x\varphi dx
-\lambda\int_{\R_{+}}W_{\alpha,B}(\pa_x\varphi)^{2}\rho^{-3}\partial_{t}\rho
dx-\int_{\R_{+}}W_{\alpha,B}\pa_x
[p(\rho)-p(\tilde{\rho})-p'(\tilde{\rho})\varphi]\frac{\pa_x\varphi}{\rho}dx\\&
-\int_{\R_{+}}W_{\alpha,B}
\frac{\pa_xp'(\tilde{\rho})}{\rho}\varphi\pa_x\varphi dx
-\int_{\R_{+}}W_{\alpha,B}u\pa_x\psi \pa_x\varphi
dx-\int_{\R_{+}}W_{\alpha,B}\frac{\tilde{u}\pa_x\tilde{u}}{\rho}\varphi
\pa_x\varphi dx\\&-\int_{\R_{+}}W_{\alpha,B}\pa_x\tilde{u}\psi\pa_x\varphi dx
-\lambda\int_{\R_{+}}W_{\alpha,B}\pa_xu\frac{(\pa_x\varphi)^{2}}{\rho^{2}}dx
-\lambda\int_{\R_{+}}W_{\alpha,B}u\frac{\pa_x\varphi\pa^{2}_x\varphi}{\rho^{2}}
dx\\&-\lambda\int_{\R_{+}}W_{\alpha,B} \frac{\pa_x\varphi}{\rho^{2}}\pa_x\rho\pa_x\psi
dx-\lambda\int_{\R_{+}}W_{\alpha,B}\pa^{2}_x\tilde{u}\varphi\frac{\pa_x\varphi}{\rho^{2}}
dx-\lambda
\int_{\R_{+}}W_{\alpha,B}\pa_x\tilde{\rho}\pa_x\psi\frac{\pa_x\varphi}{\rho^{2}}
dx\\&-\lambda\int_{\R_{+}}W_{\alpha,B}\pa^{2}_x\tilde{\rho}\psi\frac{\pa_x\varphi}{\rho^{2}}
dx-\lambda\int_{\R_{+}}W_{\alpha,B}\pa_x\tilde{u}\frac{(\pa_x\varphi)^{2}}{\rho^{2}}
dx=\sum\limits_{l=3}^{16}K_{l},
\end{aligned}
\end{equation}
where $K_l$ $(3\leq l\leq 16)$  denote the
corresponding terms on the left of \eqref{sum.d1vcf}.

Applying Sobolev's inequality, Young's inequality and
Cauchy-Schwarz's inequality with $0<\eta<1$ and using Theorem \ref{stationary}(ii), one has
\begin{equation*}
\begin{aligned}[b]
K_3=&\int_{\R_{+}}W_{\alpha,B}\pa_x\psi \pa_x(\rho u-\tilde{\rho}
\tilde{u})dx+\int_{\R_{+}}\pa_x(W_{\alpha,B})\psi \pa_x(\rho u-\tilde{\rho}
\tilde{u})dx\\
\leq&(N_{2}(T)+\eta)\|\pa_x\varphi\|_{\alpha,B}^{2}
+(C+C_{\eta})\tilde{\delta}^{2}\|[\varphi,\psi]\|_{\alpha-2,B}^{2}+(C+C_{\eta}+N_{2}(T))\|\pa_x\psi\|_{\alpha,B}^{2},
\end{aligned}
\end{equation*}
\begin{equation*}
\begin{aligned}[b]
&K_4+K_{10}+K_{11}\\=&
-\frac{\lambda|u_{b}|}{2\rho(0,t)^{2}}(\pa_x\varphi)^{2}(0,t)
+\frac{\lambda}{2}\int_{\R_{+}}W_{\alpha,B}\pa_x
\tilde{u}(\pa_x\varphi)^{2}\rho^{-2}dx\\&+\frac{\lambda}{2}\int_{\R_{+}}W_{\alpha,B}\pa_x
\psi(\pa_x\varphi)^{2}\rho^{-2}dx+\frac{\lambda}{2}\int_{\R_{+}}\pa_x(W_{\alpha,B})u(\pa_x\varphi)^{2}\rho^{-2}dx\\
\leq&-\frac{\lambda|u_{b}|}{2\rho(0,t)^{2}}(\pa_x\varphi)^{2}(0,t)
+C\tilde{\delta}\|\pa_x\varphi\|_{\alpha,B}^{2}+C\|\pa_x\psi\|_{H^{1}}\|\pa_x\varphi\|_{\alpha,B}^{2}
\\
\leq&-\frac{\lambda|u_{b}|}{2\rho(0,t)^{2}}(\pa_x\varphi)^{2}(0,t)
+C(\tilde{\delta}+N_{2}(T))\|\pa_x\varphi\|_{\alpha,B}^{2}+CN_{2}(T)\|\pa^{2}_x\psi\|_{\alpha,B}^{2},
\end{aligned}
\end{equation*}

\begin{equation*}
\begin{aligned}[b]
&|K_{7}|+|K_{12}|+|K_{14}|+|K_{16}|\\ \leq &
(\eta+C\tilde{\delta}+CN_{2}(T))\|\pa_x\varphi\|_{\alpha,B}^{2}
+(C_{\eta}+C\tilde{\delta})\|\pa_x\psi\|_{\alpha,B}^{2}
+CN_{2}(T)\|\pa^{2}_x\psi\|_{\alpha,B}^{2},
\end{aligned}
\end{equation*}
\begin{equation*}
\begin{aligned}[b]
&|K_{5}|+|K_{6}|+|K_{8}|+|K_{9}|+|K_{13}|+|K_{15}| \\ \leq&
C(N_{2}(T)+\tilde{\delta})\|\pa_x[\varphi,\psi]\|_{\alpha,B}^{2}
+C\tilde{\delta}^{2}\|[\varphi,\psi]\|_{\alpha-2,B}^{2}.
\end{aligned}
\end{equation*}

Inserting the above estimates for $K_l$ $(3\leq l\leq 16)$ into
\eqref{sum.d1vcf} and then choosing $N_{2}(T),\tilde{\delta}$ and
$\eta$ suitably small, we obtain
\begin{equation}\label{1.eng1mnl}
\begin{aligned}[b]
&\frac{d}{dt}\int_{\R_{+}}W_{\alpha,B}\left[\psi
\pa_x\varphi+\lambda\frac{(\pa_x\varphi)^{2}}{2\rho^{2}} \right]dx
+\|\pa_{x}\varphi\|_{\alpha,B}^{2}
+(\pa_x\varphi)^{2}(0,t) \\
\leq & C\|\pa_x\psi\|_{\alpha,B}^{2}
+CN_{2}(T)\|\pa^{2}_x\psi\|_{\alpha,B}^{2}+C\tilde{\delta}^{2}\|[\varphi,\psi]\|_{\alpha-2,B}^{2}.
\end{aligned}
\end{equation}

Multiplying $\eqref{1.eng1mnl}$ by $(1+t)^{\xi}$ and integrating in
$\tau$ over $[0,t]$ for any $0\leq t \leq T$, using \eqref{1mnb} and
Cauchy-Schwarz's inequality, we have the
desired estimate \eqref{2mvbc} for $\alpha \in (0,\vartheta]$.
 Then we can prove \eqref{2mvbc} holds for $\alpha=0$ with the help of Lemma
\ref{lem.V} and Sobolev's inequality.
\end{proof}

\begin{lemma}\label{lem.V12gh}
There exists a positive constant $\varepsilon_{2}$ such that if
$N_{2}(T)+\tilde{\delta}\leq\varepsilon_{2}$, then
\begin{eqnarray}\label{3mcd}
\begin{aligned}[b]
  &(1+
t)^{\xi}\| \pa_{x}[\psi,\zeta](t)\|_{\alpha,B}^{2}+\int_{0}^{t}(1+
\tau)^{\xi}\|\partial_{x}^{2}[\psi,\zeta]\|_{\alpha,B}^{2}d\tau\\
 \leq &C \left(\|[\varphi_{0},\psi_{0},\zeta_{0}]\|_{\vartheta,B}^{2}
+\|\partial_{x}[\varphi_{0},\psi_{0},\zeta_{0}]\|_{\vartheta,B}^{2}\right)\\
&
 +\xi\int_{0}^{t}(1+
\tau)^{\xi-1}\left(\|[\varphi,\psi,\zeta](\tau)\|_{\alpha,B}^{2}
+\|\partial_{x}[\varphi,\psi,\zeta]\|_{\alpha,B}^{2}\right)d\tau
\end{aligned}
\end{eqnarray}
holds for $\alpha \in [0,\vartheta]$ and $\xi\geq0$.
\end{lemma}
\begin{proof}
Multiplying $\eqref{NSP.pt}_{2}$ by $-W_{\alpha,B}\frac{\pa^{2}_x\psi}{\rho}$,
and then integrating the resulting equations over $\R_{+}$, one has
\begin{equation}\label{2.eng1xs}
\begin{aligned}[b]
&\frac{1}{2}\frac{d}{dt}\int_{\R_{+}}W_{\alpha,B}(\pa_x\psi)^{2}dx
+\lambda\int_{\R_{+}}W_{\alpha,B}\frac{(\pa^{2}_x\psi)^{2}}{\rho}dx
\\
\qquad=&\underbrace{-\int_{\R_{+}}\pa_x(W_{\alpha,B})\pa_t\psi\pa_x\psi dx}_{K_{17}}
\underbrace{\int_{\R_{+}}W_{\alpha,B}\frac{\pa_x
[p(\rho)-p(\tilde{\rho})]}{\rho}\pa^{2}_x\psi
dx}_{K_{18}} \underbrace{+\int_{\R_{+}}W_{\alpha,B}u\pa_x\psi\pa^{2}_x\psi
dx}_{K_{19}}\\&\underbrace{+\int_{\R_{+}}W_{\alpha,B}\frac{\tilde{u}\pa_x\tilde{u}}{\rho}\varphi\pa^{2}_x\psi
dx}_{K_{20}}\underbrace{+\int_{\R_{+}}W_{\alpha,B}\pa_x\tilde{u}\psi\pa^{2}_x\psi
dx}_{K_{21}}.
\end{aligned}
\end{equation}
We utilize Cauchy-Schwarz's inequality with $0<\eta<1$ to address the following estimates:
\begin{equation*}
\begin{split}
|K_{17}|\leq C\tilde{\delta}\|\pa_x[\varphi,\psi,\pa_x\psi]\|_{\alpha,B}^{2}+C\tilde{\delta}^{2}\|[\varphi,\psi]\|_{\alpha-2,B}^{2},
\end{split}
\end{equation*}
\begin{equation*}
\begin{split}
|K_{18}|+|K_{19}|\leq\eta\|\pa^{2}_x\psi\|_{\alpha,B}^{2}+C_{\eta}\|\pa_x[\varphi,\psi]\|_{\alpha,B}^{2},
\end{split}
\end{equation*}
\begin{equation*}
\begin{split}
|K_{20}|+|K_{21}|\leq\eta\|\pa^{2}_x\psi\|_{\alpha,B}^{2}+C_{\eta}\tilde{\delta}^{2}\|[\varphi,\psi]\|_{\alpha-2,B}^{2}.
\end{split}
\end{equation*}
Substituting the above estimates for $K_{l}$ $(17\leq l\leq 21)$
into \eqref{2.eng1xs} and taking $\eta$  small enough, one has
\begin{equation}\label{2.eng2asd}
\begin{aligned}[b]
&\frac{1}{2}\frac{d}{dt}\int_{\R_{+}}W_{\alpha,B}(\pa_x\psi)^{2}dx+\|\partial_{x}^{2}\psi\|_{\alpha,B}^{2}
\leq C\|\partial_{x}[\varphi,\psi]\|_{\alpha,B}^{2} +C\tilde{\delta}^{2}\|[\varphi,\psi]\|_{\alpha-2,B}^{2}.
\end{aligned}
\end{equation}

Multiplying $\eqref{NSP.pt}_{3}$ by $-W_{\alpha,B}\frac{\pa^{2}_x\zeta}{\rho}$,
and integrating the resulting equality over $\R_{+}$, we obtain
\begin{equation}\label{sum.ip5}
\begin{split}
&\frac{1}{2}\frac{d}{dt}\int_{\R_{+}}W_{\alpha,B}(\pa_x\zeta)^{2}dx+\nu\int_{\R_{+}}W_{\alpha,B}\frac{(\pa^{2}_x\zeta)^{2}}{\rho}dx
\\=&\underbrace{-\int_{\R_{+}}\pa_x(W_{\alpha,B})\pa_t\zeta\pa_x\zeta dx}_{K_{22}}\underbrace{\int_{\R_{+}}W_{\alpha,B}u\pa_x\zeta\pa^{2}_x\zeta
dx}_{K_{23}}\\&\underbrace{+\mu\int_{\R_{+}}W_{\alpha,B}\frac{\zeta}{\rho}\pa^{2}_x\zeta
dx}_{K_{24}}\underbrace{+\int_{\R_{+}}W_{\alpha,B}\frac{\pa_x\tilde{\omega}}{\rho}(\tilde{u}\varphi+\rho\psi)\pa^{2}_x\zeta dx}_{K_{25}}.
\end{split}
\end{equation}
To obtain the estimates for $K_{22}$-$K_{25}$, we use
Cauchy-Schwarz's inequality with $0<\eta<1$ to get
\begin{equation*}\label{3.3}
\begin{split}
|K_{22}| \leq
C\tilde{\delta}\|\pa_x[\zeta,\pa_x\zeta]\|_{\alpha,B}^{2}
+C\tilde{\delta}\|\zeta\|_{\alpha,B}^{2}+C\tilde{\delta}^{2}\|[\varphi,\psi]\|_{\alpha-2,B}^{2},
\end{split}
\end{equation*}
\begin{equation*}\label{3.3}
\begin{split}
|K_{23}|+|K_{24}| \leq
\eta\|\pa^{2}_x\zeta\|_{\alpha,B}^{2}+C_{\eta}\|[\zeta,\pa_x\zeta]\|_{\alpha,B}^{2},
\end{split}
\end{equation*}
\begin{equation*}
\begin{split}
|K_{25}|\leq\eta\|\pa^{2}_x\zeta\|_{\alpha,B}^{2}+C_{\eta}\tilde{\delta}^{2}\|[\varphi,\psi]\|_{\alpha-2,B}^{2}.
\end{split}
\end{equation*}

Then we have
\begin{equation}\label{2.eng2asxd}
\begin{aligned}[b]
&\frac{1}{2}\frac{d}{dt}\int_{\R_{+}}W_{\alpha,B}(\pa_x\zeta)^{2}dx+\|\partial_{x}^{2}\zeta\|_{\alpha,B}^{2}
\leq C\|[\zeta,\pa_x\zeta]\|_{\alpha,B}^{2}+C\tilde{\delta}\|\zeta\|_{\alpha,B}^{2}+C\tilde{\delta}^{2}\|[\varphi,\psi]\|_{\alpha-2,B}^{2}
\end{aligned}
\end{equation}
if $\eta$ is small enough.

The summation of $\eqref{2.eng2asd}$ and $\eqref{2.eng2asxd}$, and multiplying the resulting inequality
 by $(1+t)^{\xi}$, then integrating the resulting inequality in $\tau$
over $[0,t]$ for any $0\leq t \leq T$, using \eqref{1mnb} and \eqref{2mvbc},
   then \eqref{3mcd} holds for $\alpha \in (0,\vartheta]$.
 Then we can prove \eqref{3mcd} holds for $\alpha=0$ with the help of Lemma
\ref{lem.V} and Sobolev's inequality.
\end{proof}

\textbf{Proof of Proposition \ref{priori.estfvb}}\ \
 Now, following the three steps above, we are ready to
prove Proposition
 $\ref{priori.estfvb}$.  Summing up the estimates $\eqref{1mnb}$,
$\eqref{2mvbc}$ and $\eqref{3mcd}$, and taking $\tilde{\delta}$ and
$N_{1}(T)$ suitably small, we have
\begin{eqnarray}\label{4mvb}
\begin{aligned}[b]
&(1+
t)^{\xi}\|[\varphi,\psi,\zeta](t)\|_{\alpha,B,1}^{2}
+\int_{0}^{t}(1+ \tau)^{\xi}
\left(\tilde{\delta}^{2}\|[\varphi,\psi](\tau)\|_{\alpha-2,B}^{2}
+\|[\zeta,\partial_{x}[\varphi,\psi,\zeta],\partial_{x}^{2}[\psi,\zeta]](\tau)\|_{\alpha,B}^{2}\right)d\tau\\
\leq &C\|[\varphi_{0},\psi_{0},\zeta_{0}]\|_{\vartheta,B,1}^{2}+\xi\int_{0}^{t}(1+
\tau)^{\xi-1}\|[\varphi,\psi,\zeta](\tau)\|_{\alpha,B,1}^{2}d\tau,
\end{aligned}
\end{eqnarray}
where $C$ is a positive constant independent of $T$, $\alpha$, $\beta$,
$N_{1}(T)$ and $\tilde{\delta}$. Hence, similarly as in
\cite{S. Kawashima, M. Nishikawasll}, applying an induction to
$\eqref{4mvb}$ gives desired estimate $\eqref{3.3mnh}$.

\medskip
\noindent {\bf Acknowledgements:}
 The author was supported by the National Natural Science Foundation of China(Grant No. 11601165), the Natural Science Foundation of Fujian Province of China(Grant No. 2017J05007), and  the Scientific Research Funds of Huaqiao University (Grant No. 15BS201).

\bigbreak


\begin{thebibliography}{99}


\bibitem{Partial}
M.T. Chen, Global strong solutions for the viscous, micropolar,
compressible flow, J. Partial Differ. Equ., 24(2011), 158-164.



\bibitem{M. Chen}
M.T. Chen, Blowup criterion for viscous, compressible micropolar
fluids with vacuum,  Nonlinear Anal., Real World Appl., 13(2012),
850-859.

\bibitem{J. Zhang}
M.T. Chen, B. Huang, J.W. Zhang, Blowup criterion for the
three-dimensional equations of compressible viscous micropolar
fluids with vacuum, Nonlinear Anal., 79(2013), 1-11.


\bibitem{Commun.}
M.T. Chen, X.Y. Xu, J.W. Zhang, Global weak solutions of 3D
compressible micropolar fluids with discontinuous initial data and
vacuum, Commun. Math. Sci., 13(2015), 225-247.

\bibitem{chen1}
Q. L. Chen, C. X. Miao, Global well-posedness for the micropolar
fluid system in critical Besov spaces, J. Differential Equations, 252(2012), 2698-2724.


\bibitem{Cui1}
H.B. Cui, H.Y. Yin, Stability of the composite wave for the inflow problem on the micropolar fluid model, Commun. Pure Appl. Anal., 16(2017), 1265-1292.

\bibitem{Cui2}
H.B. Cui, H.Y. Yin, Stationary solutions to the one-dimensional micropolar fluid model in a half line: existence, stability and convergence rate, J. Math. Anal. Appl., 449(2017), 464-489.

\bibitem{Dong}
B.Q. Dong, J.N. Li, J.H. Wu, Global well-posedness and large-time decay for the 2D micropolar
equations, J. Differential Equations, 262(2017), 2488-3523.


\bibitem{spherical22}
I. Dra$\check{z}$i$\acute{c}$, N. Mujakovi$\acute{c}$, 3-D flow of a compressible viscous micropolar fluid with spherical symmetry: large time behavior of the solution, J. Math. Anal. Appl., 431(2015), 545-568.


\bibitem{spherical}
 I. Dra$\check{z}$i$\acute{c}$,  L. Sim$\check{c}$i$\acute{c}$,  N. Mujakovi$\acute{c}$, 3-D flow of a compressible viscous micropolar fluid with spherical symmetry: regularity of the solution, J. Math. Anal. Appl., 438(2016), 162-183.

\bibitem{Duancal11} R. Duan, Global solutions for a one-dimensional compressible micropolar
fluid model with zero
heat conductivity, J. Math. Anal. Appl., 463(2018), 477-495.

\bibitem{Duancal}
R. Duan, Global strong solution for initial-boundary value problem of one-dimensional compressible micropolar fluids with density dependent viscosity and temperature dependent heat conductivity, Nonlinear Anal. Real World Appl., 42(2018), 71-92.



\bibitem{A. C. Eringen} A.C. Eringen, Theory of micropolar fluids, J. Math. Mech.,
16(1966), 1-18.

\bibitem{A. C. Eringen2}
A.C. Erigen, Microcontinuum Field Theories: I. Foundations and solids, Springer. New York.,
1999.



\bibitem{HQad} F.M. Huang, X.H. Qin, Stability of boundary layer and rarefaction wave to an outflow
problem for compressible Navier-Stokes equations under large
perturbation, J. Differential Equations, 246(2009), 4077-4096.


\bibitem{L. Huang1} L. Huang, D.Y. Nie, Exponential stability for a one-dimensional compressible viscous micropolar fluid, Math. Methods Appl. Sci., 38(2015), 5197-5206.

\bibitem{JiJing}
J. Jin, R. Duan,  Stability of rarefaction waves for 1-D compressible viscous micropolar fluid model, J. Math. Anal. Appl., 450(2017), 1123-1143.

\bibitem{S. Kawashima} S. Kawashima, A. Matsumura, Asymptotic stability of traveling wave solutions of systems for
one-dimensional gas motion, Comm. Math. Phys., 101(1985), 97-127.



\bibitem{Math.} S. Kawashima, T. Nakamura, S. Nishibata, P.C. Zhu, Stationary waves
to viscous heat-conductive gases in half space: existence, stability
and convergence rate, Math. Models Methods Appl. Sci., 20(2010),
2201-2035.


\bibitem{KNZhd} S. Kawashima, S. Nishibata, P.C. Zhu, Asymptotic stability of the stationary solution to the compressible
Navier-Stokes equations in the half space, Comm. Math. Phys.,
240(2003), 483-500.

\bibitem{Qcontac}
Q.Q. Liu, H.Y. Yin, Stability of contact discontinuity for 1-D compressible viscous micropolar fluid model, Nonlinear Anal.: Theory, Methods Appl., 149(2017), 41-55.


\bibitem{Q.Q. Liu}
Q.Q. Liu, P.X. Zhang, Optimal time decay of the compressible
micropolar fluids, J. Differential Equations, 260(2016), 7634-7661.

\bibitem{ZhangLiu}
Q.Q. Liu, P.X. Zhang, Long-time behavior of solution to the compressible micropolar fluids with external force, Nonlinear Anal. Real World Appl., 40(2018), 361-376.


\bibitem{G. Lukaszewicz} G. Lukaszewicz, Micropolar Fluids. Theory and Applications, Modeling
and Simulation in Science, Engineering and Technology.
Birkh$\ddot{a}$user, Baston, 1999.

\bibitem{Matsumura1} A. Matsumura, Inflow and outflow problems in the half space for a one-dimensional
isentropic model system of compressible viscous gas, Methods Appl. Anal., 8(2001)
645-666.


\bibitem{HfklgQadcd} A. Matsumura, M. Mei,
Convergence to travelling fronts of solutions of the p-system with
viscosity in the presence of a boundary, Arch. Ration. Mech. Anal.,
146(1999), 1-22.

\bibitem{HfgQad} A. Matsumura, K. Nishihara, Large-time behaviors of
solutions to an inflow problem in the half space for a
one-dimensional system of compressible viscous gas, Comm. Math.
Phys., 222(2001), 449-474.

\bibitem{N. Mujakovi} N. Mujakovi$\acute{c}$, One-dimensional flow of a compressible viscous micropolar fluid:
 a local existence theorem, Glas. Mat., 33(1998), 71-91.

\bibitem{global}
N. Mujakovi$\acute{c}$, One-dimensional flow of a compressible
viscous micropolar fluid: a global existence theorem, Glas. Mat.,
33(1998), 199-208.


\bibitem{regularity}
N. Mujakovi$\acute{c}$, One-dimensional flow of a compressible
viscous micropolar fluid: regularity of the solution, Rad. Mat.,
10(2001), 181-193.


\bibitem{estimates05}
N. Mujakovi$\acute{c}$, Global in time estimates for one-dimensional
compressible viscous micropolar fluid model, Glas. Mat. Ser.III,
40(2005), 103-120.


\bibitem{Proceedings}
N. Mujakovi$\acute{c}$, One-dimensional flow of a compressible
viscous micropolar fluid: stabilization of the solution, in: Z.
Drma$\check{c}$, M. Maru$\check{s}i\acute{c}$, Z. Tutek (Eds.),
Proceedings of the Conference on Applied Mathematics and Scientific
Computing, Springer, Netherlands, 2005, pp. 253-262.


\bibitem{local}
 N. Mujakovi$\acute{c}$,  Nonhomogeneous boundary value problem for
one-dimensional compressible viscous micropolar fluid model: a local
existence theorem, Ann. Univ. Ferrara Sez. VII Sci. Mat., 53(2007),
 361-379.

\bibitem{boundary}
 N. Mujakovi$\acute{c}$, Nonhomogeneous boundary value problem for one-dimensional
compressible viscous micropolar fluid model: regularity of the
solution, Bound. Value Probl., 2008, Article ID 189748 (2008).


\bibitem{Math. Inequal}
N. Mujakovi$\acute{c}$, Nonhomogeneous boundary value problem for
one-dimensional compressible viscous micropolar fluid model: a
global existence theorem, Math. Inequal. Appl., 12(2009), 651-662.

\bibitem{Cauchy}
N. Mujakovi$\acute{c}$, One-dimensional compressible viscous
micropolar fluid model: stabilization of the solution for the Cauchy
problem, Bound. Value Probl., 2010, Article ID 796065 (2010).

\bibitem{Nermina}
N. Mujakovi$\acute{c}$, The existence of a global solution for one dimensional compressible viscous micropolar fluid with non-homogeneous boundary conditions for temperature, Nonlinear Anal. Real World Appl., 19(2014), 19-30.

\bibitem{NNY}
T. Nakamura,  S. Nishibata, T. Yuge, Convergence rate of solutions
toward stationary solutions to the compressible Navier-Stokes
equation in a half line, J. Differential Equations, 241(2007),
94-111.

\bibitem{M. Nishikawnj} T. Nakamura, S. Nishibata, Stationary wave associated with an inflow problem in the half
line for viscous heat-conductive gas, J. Hyperbolic Differ. Equ.,
8(2011), 651-670.


\bibitem{M. Nishikawasll} M. Nishikawa, Convergence rate to the traveling wave for viscous
conservation laws, Funkcial. Ekvac., 41(1998), 107-132.

\bibitem{Nowakowsk}
B. Nowakowski, Large time existence of strong solutions to micropolar equations in cylindrical
domains, Nonlinear Anal. Real World Appl., 14(2013), 635-660.


\bibitem{stabilization}
Y. Qin, T. Wang, G. Hu, The Cauchy problem for a 1D compressible
viscous micropolar fluid model: analysis of the stabilization and
the regularity, Nonlinear Anal., Real World Appl., 13(2012),
1010-1029.

\bibitem{wuzhigang}
Z.G. Wu, W.K. Wang, The pointwise estimates of diffusion wave of the compressible micropolar fluids, J. Differential Equations, 265(2018), 2544-2576.

\bibitem{stabstationary}
H.Y. Yin, Stability of stationary solutions for inflow problem on the micropolar fluid model, Z. Angew. Math. Phys., 68(2017), 44.

\end{thebibliography}
\end{document}